\documentclass[a4paper, 12pt]{article}

\usepackage[utf8]{inputenc}
\usepackage[english]{babel}
\usepackage[top=2cm,bottom=2cm,left=2cm,right=2cm,marginparwidth=1.75cm]{geometry}
\usepackage{amsmath}
\usepackage{amsthm}
\usepackage{amsfonts}
\usepackage{amssymb}
\usepackage{upgreek}
\usepackage{mathtools}
\usepackage{bmpsize}
\usepackage{bbm}
\usepackage{csquotes}
\usepackage{indentfirst}
\usepackage[toc,title,page]{appendix}
\usepackage{lipsum}
\usepackage[backend=biber,style=ieee-alphabetic,sorting=nyvt]{biblatex}
\usepackage[colorlinks=true, allcolors=blue]{hyperref}
\usepackage{caption}
\usepackage{xcolor}

\newcommand{\defeq}{\mathrel{\mathop:}=}
\newcommand{\eqdef}{\mathrel{\mathop=}:}

\newcommand{\mapto}{ \xrightarrow[]{}}

\newcommand{\charteq}{\overset{\cdot}{=}}
\newcommand{\eqcoord}{\dot{=}}
\newcommand{\hrz}{\mathsf{h}} 

\let\div\relax 
\DeclareMathOperator{\div}{div}
\newcommand{\dR}{\mathbb{R}}
\newcommand{\Prob}{\mathbb{P}}
\newcommand{\dE}{\mathbb{E}}

\newcommand{\ssubset}{\subset\joinrel\subset}

\newtheorem{theorem}{Theorem}[section]
\newtheorem{proposition}[theorem]{Proposition}

\newtheorem{lemma}[theorem]{Lemma}
\theoremstyle{definition}
\newtheorem{definition}[theorem]{Definition}

\newtheorem{assumption}{Assumption}
\newtheorem{assumptionalt}{Assumption}[assumption]
\newenvironment{assumptionp}[1]{
  \renewcommand\theassumptionalt{#1}
  \assumptionalt
}{\endassumptionalt}

\theoremstyle{remark}
\newtheorem*{remark}{Remark}

\theoremstyle{definition}

\newcommand*\samethanks[1][\value{footnote}]{\footnotemark[#1]}

\title{Strategic geometric graphs through mean field games}
\author{Charles Bertucci\thanks{CMAP, Ecole Polytechnique, UMR 7641, 91120 Palaiseau, France.}, Matthias Rakotomalala\samethanks}

\date{}

\begin{document}

\maketitle

\begin{abstract}
    We exploit the structure of geometric graphs on Riemannian manifolds to analyze strategic dynamic graphs at the limit, when the number of nodes tends to infinity. This framework allows to preserve intrinsic geometrical information about the limiting graph structure, such as the Ollivier curvature. After introducing the setting, we derive a mean field game system, which models a strategic equilibrium between the nodes. It has the usual structure with the distinction of being set on a manifold. Finally, we establish existence and uniqueness of solutions to the system when the Hamiltonian is quadratic for a class of non-necessarily compact Riemannian manifolds, referred to as manifolds of bounded geometry. 
\end{abstract}

\tableofcontents

\section{Introduction}
In this paper, we exploit the structure of geometric graphs on Riemannian manifolds to analyze strategic dynamic graphs at the limit, when the number of nodes tends to infinity. Namely, we propose a mathematical framework to study mean field games(MFG) in which the interaction of the agents depends on a network structure, which is itself the result of strategic interactions of the players, this will lead us to sturdy MFG systems on Riemannian manifolds.

MFG theory, introduced by Lasry and Lions \cite{lasry2007mean}, has proven to be a fruitful approach for qualitative modeling of games with a crowd of agents, such as the ones encountered in economics, finance, traffic flow, or biology \cite{carmona2020applications,cousin2011mean}. We refer to the lectures of Pierre-Louis Lions at the \textit{Collège de France}, and to the book \cite{carmona2018probabilistic} for a comprehensive introduction.
MFG in which there exists an additional structure of interaction have recently attracted attention. Caines and Huang \cite{caines2019graphon} combined graphon theory and mean field games to study games with weighted interaction between players. A \textit{graphon} is a measurable function $W$ from $[0,1]^2\mapsto [0,1]$, that represents the dense limit of a graph, where for a pair of nodes labeled by $x,y \in [0,1]$, $W(x,y)$ is the asymptotic weight of the edge between node $x$ and node $y$. Graphons are used to weigh the interaction of each agent with the rest of the crowd depending on its label, this is referred to as \textit{graphon mean field games} and there is an ongoing literature on this topic \cite{gao2020linear,lacker2023label, carmona2022stochastic, aurell2022stochastic, aurell2022finite,bayraktar2023graphon}. A different approach was proposed by Lacker and Soret in \cite{lacker2022case}, where they studied the mean-field game limit of a flocking model on \textit{transitive} graph, in this framework, it appears that the optimal control depends on the spectrum of the graph. In the previous models, the additional interaction structure, is exogenous to the game, \textit{i.e.} it is given at the beginning of the game and is, for the most part, fixed. This is mainly due to the fact that there is no \textit{a priori} canonical way to model the attachment decision of the players at the limit.

In our opinion, it is natural to consider a dynamic structure of interaction that depends on the interaction of the players themselves. In many applications, the agent controls how it interacts with the others. For example, the \textit{Lightning Network} is a decentralized payment protocol, and each pair of agents can choose to open a payment channel between them \cite{poon2016bitcoin}. Another example would be problems of mutual holding, where each bank can buy shares of the others; a mean-field approach to this problem was used by N. Touzi, F. Djete and L. Bassou in \cite{djete2021mean, djete2023mean, bassou2024mean}. While their study provides valuable insights, we provide here a modeling framework that allows to deal with other strategic structures of interaction.

The main modeling assumption we make here is that the graph that describes the interactions is a Riemannian geometric graph and that each node corresponds to a player. A geometric graph is a weighted graph, where the nodes are identified as a collection of points on a Riemannian manifold, and the edges and their weights depend uniquely on the distance of the pair of points on the manifold. We shall see that we can identify, in the limit, a geometric graph with a measure on a Riemannian manifold. There is a wide literature on geometric graphs, both in a theoretical perspective to study their properties \cite{penrose2003random,van2016random,barthelemy2011spatial,aiello2001random}, and in applications.  Krivoukov \textit{et al.} \cite{krioukov2010hyperbolic} proposed the hyperbolic space as an appropriate base space to embed real-life networks, they showed that the degree distributions of the nodes and clustering properties present similarities with network structures that can be observed in practice.
These properties motivated statistical studies, for example, in \cite{garcia2016hidden}, the authors studied the international trade maps from 1870 to 2013, and highlighted the possible hidden hyperbolic structure. Similarly in \cite{keller2021hyperbolic}, statistical properties of the European banking system point towards a hyperbolic structure. 

In our work, the geometry of the manifold that we impose to the structure of interaction is a modeling hypothesis. The main advantage of this approach is that it is now clear how to interpret the strategy of connection in this model. To change its interaction, a player simply changes its position on the Riemannian manifold and the underlying graph evolves accordingly. Note that the decision of the players depends on the position of the other players on the underlying manifold, in a usual MFG manner. Hence, the equations we shall naturally consider are of the following form,

{\small
\begin{equation*}
    \begin{cases*}
        -\partial_t v - \Delta_g v - \sigma \Delta v + H(t, x, y, \nabla_x v, \mu_t) + \Tilde{H}(t,x, y, \nabla_g v, \mu_t) = 0 \text{ on } [0,T]\times \dR^d \times M,\\
        \partial_t \mu - \div_x( D_p H (t,x,\nabla v,\mu) \mu) - \div_g( D_p \Tilde{H} (t,x, y,\nabla_g v,\mu) \mu) - \sigma \Delta \mu - \Delta_g \mu   = 0   \text{ on } [0,T]\times \dR^d \times M,\\
        \mu_{t = 0} \equiv m_0, v(T, \cdot) \equiv G(\mu_T)  \text{ on } \dR^d \times M,
    \end{cases*}
\end{equation*}
}

where $\sigma > 0$, and  $\Delta_g$, $\div_g$, $\nabla_g$ are respectively the Laplace-Beltrami operator, the divergence and the gradient operator associated to the Riemannian metric $g$ on the manifold $M$, they are the geometrical counterparts of the classical differential operators. This type of system is the main subject of study in this paper and will be derived in Section 2. It describes the evolution of a crowd of agents, represented as a density $\mu$, on the state space $\dR^d\times M$, where $(M,g)$ is a Riemannian manifold. The classical state space for mean field games is $\dR^d$ or the d-dimensional torus, we thus here extend it with $M$, which will serve as an interaction space. In this model, the relative position of the agents on the manifold serves to describe their interactions. The solution to the Hamilton-Jaccobi-Bellman equation $v$, is the value function of the optimal control faced by a typical agent, given that it anticipates the flow $(\mu_t)_{t\in [0,T]}$. In \cite{yu2022computational}, a first-order mean field games system on compact Riemannian manifolds was numerically studied, here we study a system of the second order.

An important remark is that the product of Riemannian manifold is a Riemannian manifolds, and since $\dR^d$ is a Riemannian manifold, we can study without loss of generality the system,
\begin{equation*}
    \begin{cases*}
        -\partial_t v - \Delta_\eta v + H(t, x, y, \nabla_\eta v, \mu_t) = 0\text{ on } [0,T]\times \Tilde{M},\\
        \partial_t \mu - \div_\eta( D_p H (t,x, y,\nabla_\eta v,\mu) \mu) - \Delta_\eta \mu   = 0 \text{ on } [0,T]\times \Tilde{M},\\
        \mu_{t = 0} \equiv m_0, v(T, \cdot) \equiv G(\mu_T)  \text{ on } \Tilde{M},\\
    \end{cases*}
\end{equation*}

where $\Tilde{M} = \dR^d \times M$ equipped with the product metric $\eta$. Thus in the following study, we will study the second system.

In Graphon mean-field games, one considers general uncontrolled weighted interactions between different players. In our approach, we restrict and parameterize the interaction structure through the underlying geometry of a Riemannian manifold, enabling the study of controlled dynamic interaction structures. As we shall see in Section 2, it also has the great benefit of preserving, at the limit, local geometrical information of the discrete structure, such as the Ollivier curvature. Furthermore, in many model cases, we can explicit the equation in a global coordinate system, for example, the hyperbolic case, which is a base space with interesting characteristics.

In this paper, we formulate mean-field games with a Riemannian-based interaction structure and analyze the smooth coupling case with a quadratic Hamiltonian on a manifold of bounded geometry. The technical challenges arise from the potential non-compactness of the manifold, requiring control over the behavior at infinity, particularly to obtain a comparison principle for the Hamilton-Jacobi equation on the manifold. We address this by using the parabolic results of H. Amann \cite{amann2016parabolic}, the existence and uniqueness results for stochastic differential equations from \cite{rakotomalala2024practical}, and our comparison principle. This last part depends on the existence of a family of supersolutions to an elliptic problem on the non-compact manifold, which we show holds if the manifold is of bounded geometry.

In Section 2, we introduce mean field geometric graphs, that is, limits of dynamic geometric graphs, derive the forward-backward system, corresponding to a strategic interaction, similarly as in mean field games. We also recall Ollivier coarse curvature and its convergence in the case of geometric graphs, and provide some toy models. In the third section, we establish the existence and uniqueness of second-order mean field games on manifolds of bounded geometry in the case of a quadratic Hamiltonian. This class of Riemannian manifolds encompasses compact and a wide range of non-compact manifolds, for example, the hyperbolic space, and the Euclidean $\dR^2$ classical case. Overall the strategy of proof is classical and uses a standard fixed point argument. The main novelties reside in the estimates that depend on the curvature of the manifold, we also provide proof for a maximum principle on unbounded domain for manifolds of bounded geometry needed for the analysis.

\section{Mean field geometric graph games}
In this section, we begin by recalling the notion of a geometric graph. We then define our notion of mean field geometric graph. Next, we derive the forward-backward system that describes the strategic evolution of a mean-field geometric graph. Afterward, we recall the notion of Ollivier coarse curvature and discuss its potential application in this context. Finally, we explicitly outline the equations in the model case of hyperbolic geometry, which appears suitable for certain applications.

\subsection{Strategic geometric graphs and their limits}

We begin with the usual notion of finite geometric graph.

\begin{definition}[Geometric Graph]
    Let $(\mathcal{X},d)$ be a metric space and $I = \{1,\ldots,N\}$ for some $N\in \mathbb{N}$. A finite \textit{geometric graph} on $\mathcal{X}$ is a \textit{weighted graph} $G = (V, E, w)$, associated to a triplet,
    \begin{equation*}
        ((\mathcal{X},d),\{X^i \in \mathcal{X}\}_{i \in I}, k:\dR_+\cup \{+\infty\}\xrightarrow[]{}\dR_+\cup \{+\infty\}),
    \end{equation*}
    where $V = \{X^i\}_{i\in I}$ is a collection of points taking values in $\mathcal{X}$ (the \textit{base space}),and the \textit{connecting rule} $k$ determines the edge and their weights through,
    \begin{equation*}
        E = \{(i,j)\in I^2 | k(d(X^i,X^j)) < +\infty\}, \text{ and } w_{ij} = k(d(X^i,X^j)) \text{ for } (i,j) \in E.
    \end{equation*}
\end{definition}
It is common in the geometric graph literature \cite{barthelemy2011spatial,aiello2001random} to consider connecting rules $k^\epsilon$ of the following form,
\begin{equation}
    \label{eq:DefKEps}
    k^\varepsilon(r) =  \begin{cases}
                                r \text{ if } r \leq \epsilon,\\
                                +\infty \text{ otherwise, }
                        \end{cases}
\end{equation}
that is, where each node is connected to its $\varepsilon$-neighbourhood with edge length coinciding with the base space distance.


Since we are ultimately concerned with the case of large graphs, we are lead to consider mean field limits of such finite geometric graphs. In the mean field limit, it is natural to assume that the limit graph is given by the probability measure which is obtained as the limit of the empirical measures in the finite geometric graphs. Of course, a rescaling of the connecting rule is necessary to obtain a mean field limit, but we do not enter too much in this question here, as we are more interested in the limit object and leave the question of the passage to the limit for future research. More generally, since our main concern lies in the dynamics of such graphs the connecting rule shall be of few interest for as we are going to assume that is is fixed. We shall come back on this later on.

\begin{definition}[Asymptotic Geometric Graphs]
    Given a certain connecting rule, a limit geometric graph on a base space $(\mathcal{X}, d)$ is a probability measure(or a measure) on the base space $\mu \in \mathcal{P}(\mathcal{X})$, that is the asymptotic distribution of nodes on the base space. 
\end{definition}

It could be possible to consider \textit{Gromov-Haussdorf topology} in this framework to study the limit of the graph, however, in the next example and the following sub-section, we will consider examples of scalar quantities depending on the geometry of the graph sequence, that converges at the limit, both in the sparse (when the rescaling of the connecting rule gives a constant number of neighboring nodes) and dense regimes (when the number of neighbor grows as the total number of nodes). We give a first example of \textit{dense limit}, inspired by the graphon literature.


\subsubsection{Example 1: modeling interactions with geometric graphs}

In mean field games models, given a crowd of agents identified as their state variables  $(Y^{N,i})^N_{i=1}$ in $\dR^d$, we are somtimes led to consider aggregated quantities of the form 

\begin{equation*}
    \frac{1}{N}\sum_{i=1}^N f(y,Y^i),
\end{equation*}

where $f : \dR^d \times \dR^d \mapto \dR$ is some interaction kernel. Now suppose, that we have an additional structure, in the form of a weighted graph $G = (V =\{1,\ldots,N\},E,w)$, in a similar manner it is natural to consider for an agent $i = 1, \cdots, N$, the average over the neighbors,
\begin{equation*}
    \frac{1}{\#\mathcal{N}_i} \sum_{(i,j) \in E} f(y,Y^j)w_{ij},
\end{equation*}
where $\#\mathcal{N}_i$ is the cardinal of the neighbors of $i$. It is classical, in graphon theory to introduce the weighted average, conserved under graphon convergence at the limit.
In the context of geometric graph, we consider the \textit{extended state space} $\dR^d \times \mathcal{X}$, and for a crowd $((Y^i,X^i))^N_{i=1}$. Where $Y^i$ is the classical state variable and $X^i$ is the position of the node on the manifold. This would yield a quantity of the form,
\begin{equation*}
    \frac{1}{N}\sum_{j=1}^N f(y,Y^j)\lambda^N(d(x,X^j)),
\end{equation*}
for some function $\lambda^N : \dR_+ \mapto \dR$, essentially something like $1/x$, or a possibly renormalized version(by $\frac{1}{N}\sum_{j=1}^N\lambda(d(x,X^j))$). This is the averaged quantity over the neighbors for an agent in the state $(y,x)$ given the positions of the others.
At the limit, if we suppose that we have a convergence to some probability $\nu \in \mathcal{P}(\dR^d\times \mathcal{X})$, depending on the renormalization, this could lead to,
\begin{equation*}
    \int_{\dR^k\times \mathcal{X}} f(y,y') \lambda\big(d(x,x')\big)d\nu(x', y'),
\end{equation*}
for non-renormalized dense limit, or 
\begin{equation*}
    \int_{\dR^k} f(y,y')d\nu_{Y|X=x}(y') = \dE_\nu [f(y,Y)|X = x],
\end{equation*}
for a renormalized sparse limit.


\subsection{Ollivier coarse curvature of mean field geometric graphs}
In Riemannian geometry, curvature is an analytic quantity that characterizes local and global geometric properties of the manifold. Ollivier gave a synthetic definition of curvature \cite{ollivier2009ricci}, that generalizes to metric-measured space as follows. We present this notion as we believe it is a natural object that mean field geometric graphs allow to consider.

\begin{definition}[Ollivier Coarse Curvature]
    Let $(\mathcal{X}, d, \mu)$ be a metric measured space, and assume that balls in $\mathcal{X}$ have finite measure and that $\text{supp} \mu = \mathcal{X}$. For some $\epsilon>0$, introduce the rescaled restriction of the measure to the $\epsilon$-ball around $x$, $m^\epsilon_x = \mu|_{B(x,\epsilon)}/\mu(B(x,\epsilon)).$
    
    The coarse Ollivier curvature is defined as,
    \begin{equation*}
        \kappa(x,y) \defeq 1 - \frac{W_1(m^\epsilon_x,m^\epsilon_y)}{d(x,y)},
    \end{equation*}
    where $W_1$ is the 1-Wasserstein distance or Monge-Kantorovich L1-cost(see in the appendix).
\end{definition}
One can note that, unlike Ricci curvature, the \textit{coarse} Ollivier curvature is not defined at a point, but rather between two points. In \cite{ollivier2009ricci}, Y. Ollivier establishes the coarse curvature counterpart theorems to the Bonnet-Meyer, Moser and Boltzman entropy theorems of Riemannian geometry, thus, showing that the coarse definition captures, similarly as in the Riemannian case, a variety of information about the intrinsic properties of metric-measured spaces.

If we consider a Riemannian manifold with its natural volume measure, we have for $\delta$ small enough, that, $\kappa(x,Exp_x\delta v) = \frac{\epsilon^2 \text{Ric}_x(v,v)}{2(N+2)} + o(\epsilon^3 + \epsilon^2 \delta)$, where $v \in T_xM$ is unit tangent vector, $\text{Ric}_x(v,v)$ is the Ricci curvature at $x$ in the direction $v$, $N$ is the dimension of the manifold and $Exp_x(\delta v)$ is the point at distance $\delta$ of $x$, along the geodesic issuing from $v$.

We can explicit the Ollivier curvature, in the case of a finite graph $(G=(V,E,w),d_G, \eta)$ equipped with the weighted shortest path metric, and the uniform volume measure on the nodes $\nu \in \mathcal{P}(V)$, $\eta = \frac{1}{\#V}\sum_{x \in V} \delta_x$, the coarse curvature between two connected nodes $x$ and $y$ is given by, 
\begin{equation*}
    \kappa_G(x, y) = 1 - W^1_{G}(\eta^{\epsilon}_{x},\eta^{\epsilon}_{y})/{d_{G}(x, y)}.
\end{equation*}

In Figure \ref{fig:OllivierCurvature}, we give three examples of local Ollivier curvature on the horizontal center edge.  

\begin{figure}
    \centering
    \includegraphics[width=.6\textwidth]{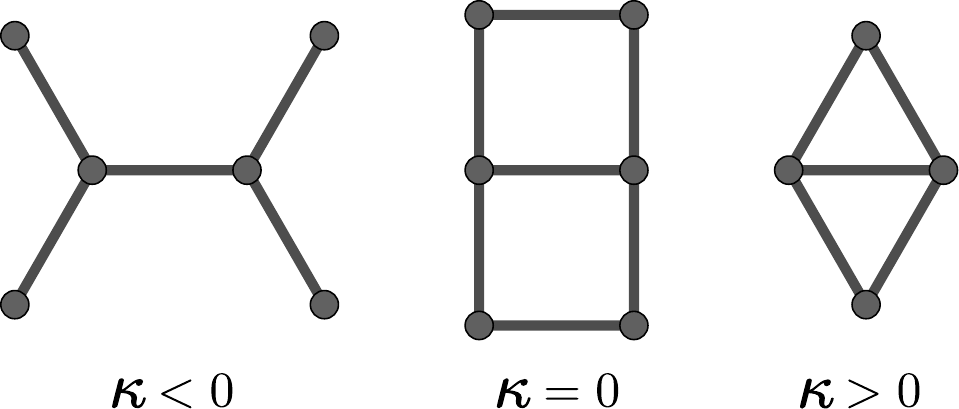}
    \caption{Coarse Ollivier curvature of the center horizontal edge of three caricature graphs, with $\varepsilon = 1$.}
    \label{fig:OllivierCurvature}
\end{figure}

We emphasize that this is an \textit{intrinsic local geometrical quantity}, that can be defined on any graph with no structure restriction. 
In recent studies,  Ollivier curvature has been used statistically to measure intrinsic properties of real-life networks; for example in \cite{sandhu2016ricci}, the average Ollivier curvature of the covariance matrix(identified as a function of agency matrix of a graph) of the stocks of the SnP500 shows to be an economic indicator that is negatively correlated with an hoarding phenomenon during crisis; as emphasized in the paper the average coarse Ollivier curvature proved to be as effective as the \textit{graph entropy}, but not like the former it is a local property of the graph, and measures locally the robustness of the network. One could find other applications in \cite{tannenbaum2015ricci} to discriminate cancerous cells, or in \cite{gosztolai2021unfolding} to identify clusters. 

Now, coming back to geometric graphs, the structure of the base space can be lifted and allows for the conservation of more goemetrical information than the one-to-one interaction in graphons theory; M. Arnaudon, X. Li, B. Petko, proved in \cite{arnaudon2023coarse} a convergence result of the rescaled Ollivier curvature in the case of geometric graphs. This result shows that, in the limit, the rescaled Ollivier curvature of a random geometric graph, drawn from $\mu$, at an ideal point $x$, will converge to the weighted Ricci curvature, 
\begin{equation*}
    \text{Ric}^\mu_{ij} = \text{Ric}^g_{ij} - \text{Hess}(\log(\mu))_{ij}.
\end{equation*}
We could thus incorporate in models the limiting curvature of an ideal node, as evoked previously this is similar to a sparse graph limit. Since the curvature is a $(0,2)-$tensor, we could consider the scalar curvature as a reward function instead, that is the contracted version with the covariant metric, this would give,
\begin{equation*}
    R^\mu \defeq g^{ij}\text{Ric}^\mu_{ij} =  R^g - \Delta_g\log(\mu),
\end{equation*}

where $R^g$ is the ambient scalar curvature of the base space manifold.

\subsection{Dynamic mean-field geometric graphs}

In numerous applications, the structure of interaction evolves over time, in the context of geometric graphs it is rather natural to define a trajectory in the space of graphs since trajectories of the nodes in the base space induce through the connection rule an evolution. We define \textit{stochastic geometric graphs} as follows. Given a filtered probability space, $(\Omega, \mathcal{F}_t, \Prob)$, with the filtration $\mathcal{F}$ that is $\Prob$-complete and a collection of $\mathcal{F}$-adapted stochastic processes $\{(X^n_t)_{t\in [0,T]}\}_{n\in I}$ taking values in $\mathcal{X}$, we can construct a stochastic geometric graph on $(\mathcal{X},d)$, by specifying a connecting rule $k$, and setting $V_t = \{X^n_t\}_{n\in I}$.

In order to define dynamical systems, we need a differential structure on the base space, thus to pursue the analysis further we restrict the metric structure to Riemannian manifolds. This still allows for a variety of geometries while enabling the use of (close-to-)classical tools of stochastic analysis, control theory, and partial differential equations. Even if we consider the base space geometry as a model specification, it is possible to restrict the study to Euclidean-geometrical graphs further, that is, to fix $\dR^d$ as the base space. However, we believe that the prism of differential geometry brings a well-suited specific treatment to the parametrization of the considered graphs. 

In order to study general situations, we want to allow the dynamics of the nodes to have a stochastic part. This naturally leads to the analysis of stochastic differential equations of Riemannian manifolds, which requires non trivial adaptation of the classical stochastic analysis theory.
It could be possible to only consider the partial differential equation associated with such stochastic differential equations. However we decided to present shortly the notion of \textit{rolling without sleeping} solution introduced by Eels and Elworthy \cite{eells1976stochastic}, since it is of great interest to study mean-field games on Riemannian manifolds from a probabilistic point of view, or simply to state the control problem of the players. This section is not required for the rest of the paper. Let $(M,g)$ be a d-dimensional Riemannian manifold, we want to define a stochastic process on $M$, associated with the equivalent of a drifted Brownian motion, that is associated with the second order operator $B\cdot \nabla_g + \Delta_g$, where $\Delta_g$ and $\nabla_g$ are respectively the Laplace-Beltrami and gradient operator. Given $f\in C^2(M)$, in charts it writes $B\cdot \nabla f + \Delta_gf \charteq b^i\partial_i f + g^{ij}\partial_{ij}f - g^{ij}\Gamma^k_{ij}\partial_kf$, where $b^i$ and $\Gamma^k_{ij}$ are respectively the coefficients of the vector field $B$ and the Christoffel symbol, using, as in the rest of the paper, the Einstein summation convention.

Formally, we want to push a process with white noise on a manifold, and define stochastic characteristics associated with the previous second-order operator. We need to specify how to map the noise onto the tangent space, \textit{a priori} there is no canonical way to define this notion, we recommend the article of Elworthy \cite{Elworthy1998}, which specifically addresses this conceptual difficulty. For the notion of \textit{rolling without sleeping} solutions, we need the following results from differential geometry,

\begin{definition}
    Let $OM$ denote the set of all orthonormal basis of the tangent space at each point of $M$,
    \begin{equation*}
        OM = \{ (x, E_1, \cdots, E_d) | x\in M, (E_1, \cdots, E_d) \text{ is an orthonormal basis of } T_xM  \},
    \end{equation*}
    and denote by $\pi : OM \mapto M$, the natural projection. Then, there exists a manifold structure on $OM$, that makes $(OM,\pi)$ a principal bundle over $M$,
    called the \textit{orthonormal frame bundle}. An element $u \in OM$, is identifiable as an isometry $u : \dR^d \mapto T_{\pi(u)}M$, for an element $\lambda \in \dR^d$, we will note $u\lambda = \sum_{i = 1}^d E_i \lambda^i \in T_{\pi(u)}M$.
\end{definition}

\begin{definition}[Horizontal Lift]
     A smooth curve $(u_t)$ taking values in $OM$, is said to be horizontal if for each $e \in \dR^d$ the vector field $(u_te)$ is parallel along the $M$-valued curve $(\pi(u_t))$. A tangent vector $Y \in T_uOM$ is said to be horizontal if it is the tangent vector of a horizontal curve at $u$. The space of horizontal vectors at $u$ is denoted by $H_uOM$; we have the decomposition
     \begin{equation*}
         T_uOM = V_uOM \oplus H_uOM,
     \end{equation*}
    where $V_uOM$ is the subspace of vertical vectors, that are tangent to the fibre $T_uOM$. It follows that the canonical projection $\pi$, induces an isomorphism $\pi_\hrz : H_uOM\mapto T_{\pi(u)}M$, and for each $B \in T_xM$ and a frame $u$ at $x$, there is a unique horizontal vector $B^\hrz$, the horizontal lift of $B$ to $u$, such that $\pi_\hrz(B^\hrz) = B$. Thus if $B$ is a vector field on $M$, then $B^\hrz$ is a vector field on $OM$.\\
    In coordinates $\{x^i,\zeta^j_k\}$, the lifted vector field writes,
    \begin{equation*}
         B^\hrz \eqcoord b^i(x)\frac{\partial}{\partial x^i} - \Gamma^k_{ij}(x)b^i(x)\zeta^j_m\frac{\partial}{\partial \zeta^k_m}.
    \end{equation*}
\end{definition}

Given $((\Omega, \mathcal{F}, \Prob), W_\cdot)$ a filtered probability space with the usual hypothesis, equipped with a $d$-dimensional Brownian motion. Let $\{e_i\}$ be the coordinate unit vectors of $\dR^d$, and $B \in C([0,T],C_b(T^1M))$ a regular vector field on $M$. For $u \in OM$, define $B^\hrz(t, u) =(B(t, \pi(u)))^\hrz$, and introduce $H_i : OM \mapto TM$ defined as $H_i(u) = ue_i$, we can consider processes $(U_\cdot)$ taking values in $OM$ solutions of  the differential equation,
\begin{equation*}
    \label{Eq:TOM}
    dU_t = \sum_{i=1}^dH^\hrz_i(U_t)\circ dW^i_t + B^\hrz(t,U_t)dt \text{ on } T_{U_t}OM, 
\end{equation*}

where $\circ$ indicates Stratonovich integral. In \cite{rakotomalala2024practical}, the second author gives an existence and uniqueness result for possibly stochastic drift and a regular diffusion coefficient on Riemanian manifolds of bounded geometry.

One can note that, since any element of $OM$ can be identified as a point $x$ in $M$ and an associated orthonormal basis of $T_xM$, equation \eqref{Eq:TOM} describes a moving frame along a process $X_t = \pi(U_t)$ taking values in $M$, itself solution of :
\begin{equation}
    \label{eq:XSDequation}
    dX_t = \sum_{i=1}^d (U_te_i)\circ dW^i_t + B(t,X_t)dt \text{ on } T_{X_t}M.
\end{equation}
The fact that the equation on $X$ is not autonomous but involves $U$, translates the need to keep track of how the noise should push the process $U$ on the manifold, that is how the noise pushes the orientation of $T_UM$.

\begin{remark}
We need Stratonovich integral to give a chart-independent definition of a solution. Indeed, without the chain rule given by Stratonovich integral, we would get a second-order term when changing chart, corresponding to Ito's corrective term, that wouldn't correspond to the change of coordinate of a tensor field, and thus the process wouldn't be a solution of the same equation by change of coordinate.
\end{remark}

The Fokker-Planck equation is obtained, similarly as in the Euclidian case, by applying Ito's formula to a test function $\varphi \in C_{1,2}([0,T] \times M)$, \textit{i.e }

\begin{equation*}
    \dE[\varphi(t,X_t) - \varphi(s, X_s)] = \dE\Big[\int_s^t (\partial_t \varphi + B \cdot \nabla_g \varphi + \frac{1}{2}\Delta_g \varphi) (v, X_v)dv\Big].
\end{equation*}

Supposing that the process $(X_\cdot)$ has regular a density $\mu \in C^{1,2}([0,T]\times M)$,

\begin{equation*}
    \int_M \varphi(t)\mu_t - \int_M \varphi(s)\mu_s = \int_s^t \int_M (\partial_t \varphi + B\cdot \nabla_g \varphi - \frac{1}{2} \Delta_g \varphi)\mu_v dv,
\end{equation*}

where the integral is with respect to $dvol_M$ the manifold volume associated to the metric $g$. Integrating by part in $t$ and $x$, we get,

\begin{equation*}
    0 = \int_s^t \int_M (\partial_t \mu_v + \div_g(B\mu_v) - \frac{1}{2} \Delta_g \mu_v)\varphi dv.
\end{equation*}

 This naturally gives the Fokker-Planck equation satisfied by the law of the solution of the stochastic differential equation,

\begin{equation}
    \label{eq:FPKequation}
    \partial_t \mu + \div_g(B\mu) - \frac{1}{2} \Delta_g \mu = 0.
\end{equation}

Hence, if the dynamics of the nodes are given by \eqref{eq:XSDequation}, then the graph evolves according to \eqref{eq:FPKequation}.
\subsection{Deriving the mean field game system}

We will now derive the associated mean field game system, which models the strategic equilibrium in a MFG in which players control their interaction with the graph.

To introduce controlled systems in differential geometry, we need the concept of \textit{vector bundles}(definition in the appendix). Formally, a vector bundle is a varying family of vector spaces on $M$. For example, this concept addresses the fact that in the case of the vector space of infinitesimal movements at some point, referred to as the \textit{tangent space} at that point, does not coincide with the tangent space at another point. A \textit{section} of the tangent space is a vector field on $M$. For instance, on the embedded 2-sphere in $\dR^3$, a section of the tangent space is a continuously varying family of vectors over the sphere, where each vector associated to a point of the sphere is in the tangent plane of the sphere at that point. 

Let $E$ be a vector bundle on $M$, and $C(E)$ the space of global continuous sections of $E$. A instantaneous control will take value in $C(E)$, for example :
\begin{enumerate}
    \item $E = M \times \dR^k$ trivial bundle. 
    \item $E = TM$, the tangent bundle, defined as the union of all the tangent spaces at each point $TM \defeq \cup_{x\in M} T_xM$, allows to control the speed(see Lagrangian mechanics on Manifolds in \cite{fathi2008weak}).
\end{enumerate}

The set of admissible controls is defined as,
\begin{equation*}
    \mathcal{U}_t = \{ \alpha : [t,T) \xrightarrow{} C(E) , \alpha \text{ is }\mathcal{F-}\textit{adapted} \}.
\end{equation*}

We now can introduce \textit{strategic graphs} within the framework of mean field games and geometrical graphs. Suppose that some agent faces, over a time horizon $[0,T]$, a 'mean field geometrical graph' on some manifolds, represented as a probability flow $(\mu_t)_{t\in[0,T]}$ in $C([0,T],\mathcal{P}(M))$. The agent controls the position of a node; and is rewarded for its contribution in the geometrical structure. Given a running and a terminal cost functional $L : [0,T) \times \mathcal{P}(M) \times E \xrightarrow{} \dR $ and $G : \mathcal{P}(M) \times M \xrightarrow{} \dR $, we can represent the optimal control problem encountered by the agent as,
\begin{equation}
    \label{sys:StoOpContProblem}
    \begin{cases}
        V(t,x) = \underset{\nu \in \mathcal{U}_t}{\inf}\dE \big[ \int_t^T L(s, \mu_s, X_s, \nu_s) ds + G(X_T,\mu_T) | X_t = x], \\
        dX_t = \sum_{i = 1}^d (U_te_i)\circ dW^{i}_t + B(t,X_t,\nu_t)dt \text{ on } T_{X_t}M,
    \end{cases}
\end{equation}
where $B: [0,T] \times M \times C(E) \mapto TM$ is the controlled drift. $V$ is called the value function of the control problem, assuming it is regular, from a dynamic programming principle and Ito's formula, similarly as in the classical case we obtain the following Hamilton-Jacobi-Bellman equation that be satisfied by $V$,

\begin{equation*}
    \begin{cases*}
        \partial_t V  + \Delta_g V + H(t, x, \nabla_g V, \mu_t) = 0 \text{ on }[0,T)\times M,\\
        V(T, \cdot) \equiv G(\mu_T),
    \end{cases*}
\end{equation*}

where $H : [0,T] \times TM \times \mathcal{P}(M) \mapto \dR$ is the associated Hamitonian defined as, 

\begin{equation*}
    H(t,x, p , m) = \inf_{\alpha \in E_x} \{ B(t,x, \alpha) \cdot p + F(t,x, m , \alpha) \}.
\end{equation*}

If we suppose further that the Hamiltonian is regular and the value function is smooth, the optimal closed loop drift is given by $ - D_p H(t, X^*_t,\nabla_g V(t,X^*_t), \mu_t)$, with $D_pH : [0,T] \times TM \times \mathcal{P}(M) \mapto TM$, being the derivative with respect to the third argument.
 
In mean field games, the Nash equilibrium, is represented as a probability flow $(\mu_t)_t$, characterized as a fixed point as, $\mu_t = \mathcal{L}(X^*_t)$, where $(X^*_{\cdot})$ is the optimal controlled process of problem \eqref{sys:StoOpContProblem}, with the probability flow $(\mu_t)_t$. This leads to the following partial differential equation system,

\begin{equation}
    \begin{cases*}
        -\partial_t v - \frac{1}{2}\Delta_g v + H(t,x, \nabla_g v, \mu_t) = 0 \text{ on } [0,T)\times M,\\
        \partial_t \mu - \div_g( D_p H (t,x,\nabla_g v,\mu) \mu) - \Delta_g \mu  = 0  \text{ on } [0,T)\times M,\\
        \mu_{t = 0} \equiv m_0, v(T, \cdot) \equiv G(\mu_T).
    \end{cases*}
\end{equation}

The first equation corresponds to the Hamilton-Jacobi-Bellman equation, associated with the value function of the game $V$. The second equation, the Fokker-Planck equation, describes the evolution of the density of nodes. The boundary conditions correspond to the initial distribution of players and the terminal cost. Some models lead to separated Hamiltonians, and thus to the following system,

\begin{equation}
    \label{sys:GeneralMFGsys}
    \tag{$MFG_{sys}$}
    \begin{cases*}
        -\partial_t v - \frac{1}{2}\Delta_g v + H(t,x, \nabla_g v) = F(t,x,\mu_t) \text{ on } [0,T)\times M,\\
        \partial_t \mu - \text{div}_g( D_p H (t,x,\nabla_g v) \mu) - \Delta_g \mu  = 0  \text{ on } [0,T)\times M,\\
        \mu_{t = 0} \equiv m_0, v(T, \cdot) \equiv G(\mu_T),
    \end{cases*}
\end{equation}

where $F$ is the running cost function. Similarly as in the classical case, we can apply Lasry and Lions \cite{lasry2007mean} monotonicity hypothesis to prove the uniqueness of this system. 

\begin{definition}(Monotone)
    $F : M \times \mathcal{P}(M) \xrightarrow[]{} \dR$, is said to be monotone if,
    \begin{equation*}
        \int_M (F(x,\mu) - F(x, \nu))(\mu(dx) - \nu(dx)) \geq 0, \forall \mu, \nu \in \mathcal{P}(M).
    \end{equation*}
    Furthermore, F is said to be strictly monotone if,
    \begin{equation*}
        \int_M (F(x,\mu) - F(x, \nu))(\mu(dx) - \nu(dx)) > 0, \forall \mu, \nu \in \mathcal{P}(M), \mu \neq \nu.
    \end{equation*}
\end{definition}

\begin{theorem}
   Suppose that $G$ is monotone, $F$ strictly monotone and that $H$ is convex in the $p$ variable, then the solution to \ref{sys:GeneralMFGsys} is unique.
\end{theorem}

\begin{proof}
    Let $(u^i,m^i), i = 1,2$ be two solutions, denote by $\Bar{m} = m^1 - m^2$ and $\Bar{u} = u^1 - u^2$, then using the Fokker-Planck equation with $\Bar{u}$ as a test function, we get,
    \begin{align*}
        \int_0^T\int_M (\Bar{m}(\partial_t\Bar{u} + \Delta \Bar{u}) & + \nabla\Bar{u}\cdot (D_pH(\nabla u^1)m^1 - D_pH(\nabla u^2)m^2))dxdt \\  & = \underset{\geq 0}{\underbrace{\int_M (G(m^1_T)-G(m^2_T))\Bar{m}_T(dx)}}- \underset{=0}{\underbrace{\int_M \Bar{u}(0)\Bar{m}_0(dx)}} \geq 0.
    \end{align*}

    Now using the Hamilton-Jacobi-Bellman equation, one gets,
    
    \begin{align*}
        -\int_0^T\int_M ((F(m^1_t)-F(m^2_t))\Bar{m}_t(dx)dt \geq &-\int_0^T \int_M \Bar{m}(H(\nabla u^1) -H(\nabla u^2)) dxdt \\ & + \int_0^T \int_M \nabla\Bar{u}\cdot (D_pH(\nabla u^1)m^1 - D_pH(\nabla u^2)m^2))dxdt 
    \end{align*}

    Noting that from the convexity of $H$,
    \begin{equation*}
        H(t,x,q) - H(t,x,p) \leq D_pH(t,x,q)\cdot (q-p), \ \forall t\in [0,T],\forall x \in M, \forall p,q \in T_xM.
    \end{equation*}

    This implies that the right and side of the last inequality is positive, so that

    \begin{equation*}
        \int_0^T\int_M ((F(m^1_t)-F(m^2_t))\Bar{m}_t(dx)dt \leq 0.
    \end{equation*}

    From the strict monotonicity of $F$ for any $t$ this implies that $m^1 = m^2$.
\end{proof}
\subsection{The case of hyperbolic geometry}

The hyperbolic space proves to be an adequate base space to reproduce real-life network properties, from the seminal work of \cite{krioukov2010hyperbolic} there has been an ongoing literature on hyperbolic graph models, one can consult \cite{aiello2001random, garcia2016hidden,aldecoa2015hyperbolic}. From a modelization point of view, it comes with the benefits of having models with global charts, the study can thus be reduced to classical forward-backward systems in domains of $\dR^n$ with coefficients issuing from the metric tensor in the global chart. We propose to expand the equations in the Poincaré disk model: $(B_{\dR^n}(0,1), \frac{4\delta_{ij}}{(1-\|x\|^2)^2})$. 

The domain of the chart is the unit ball of $\dR^n$, and it has a conformal metric defined as,
\begin{equation*}
    g_{ij}(x) = 4\delta_{ij}/(1-\|x\|^2)^2,
\end{equation*}
where $\delta_{ij}$ is the Kronecker delta symbol.
Here $\nabla, \Delta, div, \langle \cdot , \cdot \rangle, \|\cdot\|$ respectively correspond to the gradient, the Laplacian, divergence, scalar product and norm of $\dR^n$. By definition, the Laplace-Beltrami operator writes,
\begin{equation*}
    \Delta_g u = \frac{1}{\sqrt{G}}\partial_i (\sqrt{G}g^{ij}\partial_j u ) =\frac{1}{4} (1-\|x\|^2)^n\sum_{i=0}^n \partial_i ((1-\|x\|^2)^{2-n}\partial_i u ),
\end{equation*}

developing one gets,
\begin{equation*}
    \Delta_g u = \frac{(n-2)(1-\|x\|^2)}{2}\langle x, \nabla u \rangle + \frac{(1-\|x\|^2)^2}{4} \Delta u.
\end{equation*}

The norm of the gradient writes,
\begin{equation*}
    \|\nabla_g u \|_g^2 = \partial_i u g^{ij} \partial_j u = \frac{1}{4}(1-\|x\|^2)^2\sum (\partial_i u )^2 = \frac{1}{4} (1-\|x\|^2)^2\|\nabla u\|^2.
\end{equation*}

So that mean field game system with quadratic Hamiltonian on the n-dimensional hyperbolic space writes as,
{\small
\begin{equation*}
    \begin{cases*}
        -\partial_t u - \frac{(n-2)(1-\|x\|^2)}{4}\langle x, \nabla u \rangle - \frac{(1-\|x\|^2)^2}{8} \Delta u + \frac{(1-\|x\|^2)^2}{8}\|\nabla u \|^2  = F(m_t, x) \text{ \ \ on } [0,T] \times B_{\dR^n}(0,1),\\ %
        \partial_t m - \frac{(n-2)(1-\|x\|^2)}{4}\langle x, \nabla m \rangle - \frac{(1-\|x\|^2)^2}{8} \Delta m + \frac{1}{4}(1-\|x\|^2)^2 \text{div}(\nabla u m) + \frac{(n-2)}{2}(1-\|x\|^2)\langle x, \nabla u \rangle m = 0.
    \end{cases*}
\end{equation*}
}

One can check that the operator of the Hamilton-Jacobi-Bellman equation satisfies an invariance domain condition\cite{cannarsa2010invariant}, however, the operator of the Fokker Planck equation is not the adjoint of the linearised with respect to the Lebesgue measure, it is true w.r.t to the volume measure of the manifold $dvol = 2dx/(1-\|x\|^2)^n$. One should note that in this case, we could study the  system following the work of \cite{porretta2020mean}, since the metric gives an invariance condition for the domain, with the difference that the Fokker-Planck equation is satisfied w.r.t the volume measure induced by $g$, and not the Lebesgue measure.






\subsection{An instructive example: Curvature mean-field game}
    We introduce here the example of the Curvature mean-field game which we believe to be of interest and we study it on a compact manifold with linear quadratic cost in the stationary case. Inspired by the asymptotic Ollivier curvature, we propose the following mean field \textit{toy model}. 
    
    Let $(M,g)$ be a compact manifold, fix some discount factor $r>0$, and consider the stationary mean field game of the form,
    \begin{equation}
        \label{sys:CurvatureMFG}
        \begin{cases}
            0 = \text{div}(\nabla v m) + \Delta_g m \text{ on } M,\\
            0 = -rv - \frac{1}{2} |\nabla v|^2_g + \Delta_g v + R^m\text{ on } M,
        \end{cases}
    \end{equation}

    where $R^m$ is the mean-field scalar curvature, defined in the previous section, as $g^{ij}\text{Ric}^m_{ij} =  R^g - \Delta_g\log(m)$, with $R^g$ the scalar ambient curvature.
    
    This system is interpreted as the equilibrium of an infinite horizon stochastic linear quadratic mean-field game, wherein the cost function corresponds to the local average Ollivier curvature of the limit of a large graph formed by the aggregated strategies of the agents. Each agent controls its speed and minimizes a kinetic energy term plus its average curvature.
    
    The agents aim to minimize curvature since the discrete setting asserts that agents with negative curvature have neighbors who depend on them to stay connected to the rest of the graph, thus being central in the local connection structure. Unlike a centrally planned network, designed for \textit{robustness} with redundant connections and positive average Ollivier curvature, competitive agents strategically shape their environments to create \textit{local bottlenecks}. This enforces their neighbors' dependency on them, resulting in tree-like, hierarchical structures with negative curvature. For example, negative curvature in a non-centrally planned network is measured on the Internet network \cite{ni2015ricci}.

    System~\eqref{sys:CurvatureMFG} can be reduced to a single nonlinear equation in the following manner.
    We define a measure $m$ as, 
    \begin{equation*}
        m = e^{-v}dvol_M.
    \end{equation*}
    Since $M$ is compact, and assuming that $v$ is twice continuously differentiable, then $m$ is a finite-mass measure solution of the first equation, so up to a renormalization, it is a probability solution to the Fokker-Planck stationary equation.
    Plugging this $m$ in the second equation,
    \begin{equation*}
        -rv - \frac{1}{2} |\nabla v|^2_g + \Delta_g v + \underset{R^\mu}{\underbrace{R^g + 2\Delta_g v}} = -rv - \frac{1}{2} |\nabla v|^2_g + 3\Delta_g v + R^g = 0.
    \end{equation*}
    Thus, the system reduces to a quasi-linear elliptic equation.
    \begin{equation*}
        \begin{cases}
            m = e^{-v}dvol_M\frac{1}{\int e^{-v}dvol_M}, \\
            -rv - \frac{1}{2} |\nabla v|^2_g + 3\Delta_g v + R^g = 0.
        \end{cases}
    \end{equation*}
    
    If the scalar curvature is constant, then the solution is explicitly given by $v \equiv \frac{1}{r}R^g$ and $m$ is the uniform measure on the manifold, there is no preferred position. In the general case, the mean-field curvature $R^m$ tends to be an averaged version of the scalar curvature of the base space. Near equilibrium, the nodes are attracted towards minima of the scalar curvature, causing a concentration of the measure in that area and thus increasing the local scalar curvature, turning the neighborhood into a less desirable target.

    The main motivation is to give meaning to dynamic graphs and controlled geometries. This example hints that geometric graphs provide a simple framework that allows for computations, while still preserving some geometric information at the limit. We believe that this represents a first step in the modeling process and that geometric graphs offer a wide range of tools to introduce complexity into future models.


\section{A case study: mean field games systems on Riemannian manifolds of bounded geometry with quadratic Hamiltonian}
In this section, we prove existence in the case of a quadratic Hamiltonian on (possibly) non-compact class of Riemannian manifold. We restrict the model system \eqref{sys:GeneralMFGsys} so that each player controls its speed, with $E = TM$. We consider the separated cost of the form: $F(t, \mu, x, \alpha) \defeq \Tilde{F}(\mu, x) + \frac{1}{2}\|\alpha\|^2_g$. Thus, for $(x,p)\in TM$, that is $x\in M, p \in T_xM$, the associated Hamiltonian is of the form,
\begin{equation*}
    H(x, p , \mu) = \inf_{\alpha \in T_xM} \Big\{ \langle \alpha, p \rangle_{g(x)} + \frac{1}{2} \|\alpha\|^2_{g(x)} \Big\} + \Tilde{F}(\mu, x),
\end{equation*}
using a local chart and writing,
\begin{equation*}
    \inf_{\alpha \in \dR^d} \Big\{ \alpha^i g_{ij} p^j + \frac{1}{2}\alpha^ig_{ij}\alpha^j \Big\} + \Tilde{F}(\mu, x),
\end{equation*}
one concludes that, the infimum is $-\frac{1}{2}p^ig_{ij}(x)p^j + \Tilde{F}(\mu, x)$, thus,
\begin{equation*}
    H(x, p, \mu) = -\frac{1}{2} \|p\|^2_g + \Tilde{F}(\mu, x),
\end{equation*}
with the optimal $\alpha$ given by $D_p H(x,p,\mu) = -p$. Finally, we obtain the following system,
\begin{equation*}
    \label{eq:MFG}
    \begin{cases*}
        -\partial_t v - \Delta_g v + \frac{1}{2}|\nabla v|_g^2 = F(m_t) \text{ on } [0,T]\times M,\\
        \partial_t m - \text{div}_g( \nabla v m) - \Delta_g m  = 0 \text{ on } [0,T]\times M,\\
        m_{t = 0} = m_0, v(T, \cdot) = G(m_T).
        \tag{QuadSys}
    \end{cases*}
\end{equation*}

Although the structure of the equation is similar to a classical linear quadratic mean field game system, the estimates require a specific treatment of the behavior of solutions at infinity, as we aim to treat non-compact manifolds, such as hyperbolic space. This will be addressed by assuming the manifold is of bounded geometry, encompassing, for example, the compact case, Euclidean space, and hyperbolic space. This will allow us to apply the results from \cite{amann2016parabolic, rakotomalala2024practical} and establish a comparison principle for the Hamilton-Jacobi-Bellman equation.
\subsection{Riemannian Manifolds of bounded geometry}
In order to study the system we need to introduce the hypotheses on the geometry of the manifold.
\begin{definition}
    Let $i : M \to \dR_+$ be the injectivity radius function, that is $i(x)$ is the largest radius for which the exponential map $\exp_x$ is a diffeomorphism. A manifold is said to have a positive injectivity radius if,
    \begin{equation*}
        \underset{x \in M}{\inf} i(x) > 0.
    \end{equation*}
\end{definition}

\begin{definition}[Manifold of bounded Geometry]
    A Riemannian manifold $(M,g)$ equipped with its Levi-Civita connection is said to be \textit{of bounded geometry} if, it is a \textit{complete} metric space, it has \textit{positive injectivity radius}, and if all covariant derivatives of the Riemann curvature tensor are bounded,
    \label{def:BoundedGeo}  
    \begin{equation*}
            \| |\nabla^k R |_g\|_\infty \leq C(k), \forall k \in \mathbb{N}_0.
    \end{equation*}
\end{definition}

In the following, we will use Herbert Amann \cite{amann2016parabolic} result of maximal regularity for parabolic equations on \textit{uniformly regular manifold}. A notable result proven by M. Disconzi, Y. Shao, and G. Simonett in \cite{disconzi2014some}, is the equivalence between the geometrical definition of manifold of bounded geometry and the notion of uniformly regular manifold. A few examples of manifolds of bounded geometry, are any Euclidean space, any isometric images of uniformly regular Riemannian manifolds, any compact Manifold, the d-dimensional hyperbolic space or the d-dimensional sphere.

    

\subsection{Parabolic equations on Riemannian manifolds}\label{subsec:ParabolicEquationCalpha}
We now state a maximal regularity result for linear parabolic equations on manifolds due to Herbert Amann \cite[Theorem 1.23]{amann2017cauchy}, needed for the analysis. To this end, we define the following function spaces.
For $\sigma, \tau \in \mathbb{N}$, let $V = T^\sigma_\tau M$ be a tensor bundle on M. For $k \in \mathbb{N}$, we denote by $C^k(V)$ the space of $k$-derivable sections of $V$, and $C^k_b(V)$ the subspace of bounded section with bounded derivatives, it is a Banach space with the norm $\|\cdot\|_{k,\infty}$, defined for $u \in C^k_b(V)$ as,
\begin{equation*}
    \|u \|_{k,\infty} \defeq \max_{0\leq i \leq k} \| |\nabla^i u|_g\|_\infty
\end{equation*}
where $|\cdot|_g$ denotes the complete contraction induced by the metric $g$.
Now, let $\alpha \in (0,1)$, we define the $(k+\alpha)$-\textit{little Holder} space as,
\begin{equation*}
    C_b^{k+\alpha}(V) \defeq (C^k_b(V), C^{k+1}(V))^0_{\alpha,\infty},
\end{equation*}
where $(\cdot, \cdot)^0_{\alpha,\infty} $ is the continuous interpolation method \cite{amann2017cauchy,amann2013function}.
Now we introduce the parabolic anisotropic function spaces \cite{amann2012anisotropic}, for $T\in \dR$, let $J = [0,T]$,
\begin{equation*}
    C_b^{((k+\alpha)/2,k+\alpha)}(J\times V) = BUC(J,C_b^{k+\alpha}(V))\cap C^{(k+\alpha)/2}(J,C_b(V)),
\end{equation*}
where $BUC$ is the space of \textit{bounded uniformly continuous} functions.

We have the following equivalence,
\begin{align*}
    u \in C_b^{( (k+\alpha)/2 + 1, k+\alpha + 2)}(J\times V) \text{ iff } & \\ \nabla^j u \in C_b^{((k+\alpha)/2, k+\alpha)}(J\times V^\sigma_{\tau + j})
    & \text{ for } 0 \leq j \leq 2 \text{ and } \partial_t u \in C_b^{((k+\alpha)/2,k+\alpha)}(J\times V)
\end{align*}

\begin{proposition}
\label{th:parabolicReg}
Let $(M,g)$ be a manifold of bounded geometry, $J = [0,T]$, $k\in \mathbb{N}$ and $0< \alpha <1$.
Take $B \in  C_b^{k+\alpha,(k+\alpha)/2}(J \times TM), a \in C_b^{k+\alpha,(k+\alpha)/2}(J \times M),f\in C_b^{k+2+\alpha,(k+2+\alpha)/2}(J \times M)$ and $ \varphi \in C_b^{k+2+\alpha}(M)$.
Then, there exists a unique solution $u$ belonging to the space $C_b^{k+2+\alpha,(k+2+\alpha)/2}(J \times M)$, solution of the backward parabolic problem,
 \begin{equation*}
     \begin{cases*}
         \partial_t u + B\cdot \nabla_g u + \Delta_g u = f \text{ on } (0,T)\times M, \\
         u_{t = T}  = \varphi \text{ on } M.
     \end{cases*}
 \end{equation*}

    Furthermore, we have the estimate,
    \begin{equation*}
        \|u\|_{C_b^{k+2+\alpha,(k+2+\alpha)/2}}\leq C(\|f\|_{C_b^{k+2+\alpha,(k+2+\alpha)/2}}+ \|\varphi\|_{C_b^{k+2+\alpha}}),
    \end{equation*}
    where $C$ depends on geometrical bounds, the norm of the coefficients, and on T only.
\end{proposition}

\begin{remark}
    The more general theorem holds true for $r-$th order operators defined on $(\sigma, \tau)-$tensors fields.
\end{remark}

\subsection{Maximum principle for parabolic equations.}\label{subsec:MaxPrinPara}
For the upcoming analysis, we require a maximum principle. On a Riemannian manifold, at a maximal point of some sub-solution to a parabolic equation, one can conclude, using a local chart, a similar contradiction as in the classical maximum principle. In an unbounded domain, one must employ an auxiliary function to establish the existence of such a maximizer. In the case of $\dR^d$, one can use $x\mapsto \frac{\epsilon}{2}|x|^2$, to obtain a maximizer of $u(x) - \frac{\epsilon}{2}|x|^2$, and conclude a contradiction for any $\epsilon > 0$. However, on a general Riemannian manifold, the distance function to a reference point squared might not be twice differentiable everywhere. A simple example is the infinite cylinder: $  \mathbb{T} \times \dR$, where $(x,y) \in [0,1]\times \dR \mapsto d^2((0,0),(x,y))$ is not differentiable along $\{(1/2,y), y\in \dR\}$, essentially the distance function will not be differentiable where there is more than one geodesic connecting the two points. We give in Proposition \ref{prop:ComparaisonPrinciple} several conditions on the Riemannian manifold under which the maximum principle holds.

\begin{definition}
    \label{def:ComparaisonPrinciple}
    We say that $(M,g)$ satisfies the parabolic maximum principle if, for any bounded $u$ in $C^{1,2}([0,T]\times M)$, sub-solution of the parabolic equation, 
    \begin{equation*}
        -\partial_t u -\Delta_g u + ru \leq 0 \text{ in } \mathring{M_T},
    \end{equation*}
    with $\underset{(t,x)\in M_T}{\inf}r(t,x) > 0$, and $M_T = [0,T]\times M$. We have,
    \begin{equation*}
        u(T, \cdot) \leq 0 \text{ in } M \implies  u \leq 0 \text{ in } M_T
    \end{equation*}

\end{definition}
\begin{proposition}
    \label{prop:ComparaisonPrinciple}
    Each of the following conditions are sufficient for a Riemannian manifold to satisfy a parabolic maximum principle.
    \begin{enumerate}
        \item If $(M,g)$ is compact.
        \item $\forall \lambda > 0 , \exists \psi^\lambda \in C^2(M)$, such that $ \lim_{d(x_0, x) \xrightarrow[]{} +\infty } \psi^\lambda(x) = + \infty $ for some $x_0 \in M$, and  $\psi^\lambda$ is a super-solution of the elliptic equation,
            \begin{equation*}
                    \tag{$\Psi$}
                    \label{hyp:EllipCoerc}
                    0 \leq \lambda \psi^\lambda - \Delta_g \psi^\lambda \text{ in } M.
            \end{equation*}
        \item If $(M,g)$ is connected and has negative sectional curvature, and the Ricci curvature is bounded below.
        \item If $(M,g)$ is connected and of bounded geometry.
    \end{enumerate}
\end{proposition}

It is noteworthy that the third assumption is closely related to the one used in stochastic theory on manifolds to prove global existence results, as demonstrated in \cite{elworthy1982stochastic}. Moreover, to prove the second and fourth points, we will construct such super-elliptic functions on the manifold. For the proof of this result, we need to use the following two results.

\begin{lemma}[\cite{shubin1992spectral} Lemma 2.1, p.70]
    \label{lem:regulDist}
    Suppose that $(M,g)$ is a \textit{connected} manifold of bounded geometry. There exists a function $\Tilde{d} : M\times M \mapto [0,\infty)$ satisfying the following conditions,
    \begin{enumerate}
        \item there exists $\epsilon > 0$ such that:
            \begin{equation*}
                |\Tilde{d}(x,y) - d(x,y)|< \epsilon.
            \end{equation*}
        \item $\forall x \in M$, $\Tilde{r}_x : y \mapsto \Tilde{d}(x,y)$ is $C^\infty(M)$, and its derivatives are uniformly bounded in $x$, in other words, for $n > 0$,
            \begin{equation*}
                \|\nabla^n \Tilde{r}_x\|_{\infty} \leq C(n) \ \ \ \  \forall x \in M.
            \end{equation*}
    \end{enumerate}    
\end{lemma}

This result is due to Yu.A. Kordyukov and relies on an appropriate partition on the unity to regularise the distance function.
There exists a proof in English, in \cite{shubin1992spectral}, where the statement give a uniform bound in coordinates of $|\partial^\alpha_y \Tilde{d}(x,y)|$, to conclude one uses the uniform equivalence of the euclidean norm in a local chart and the pullback metric. Finally, in order to prove the maximum principle we recall the following result.

\begin{proposition}[Comparison Theorem for the Laplacian, \cite{hsu1988brownian} Theorem 3.4.2, p.90]
    \label{thm:LaplacianDistBoundedCurve}
    Let $(M,g)$ be a $n$-dimensional Riemannian manifold and $x_0 \in M$ and note $r(x) = d(x_0, x)$, suppose that the sectional curvature is bounded from above by $K_1^2$ and the Ricci curvature is bounded from below by $-(n-1)K_2^2$. Then inside the cut locus of $x_0$, 
    \begin{equation*}
        (n-1)K_1cot (K_1 r(x)) \leq \Delta_g r(x) \leq (n-1) K_2 coth (K_2 r(x)).
    \end{equation*}
\end{proposition}

\begin{proof}[Proof of Proposition \ref{prop:ComparaisonPrinciple}]
    \begin{enumerate}
        \item There exists a maximizer of $u$ thus, it is a straightforward application of the maximum principle.
        \item 
        Take $u$ a bounded $C^{1,2}$ sub-solutions of the parabolic equation, satisfying \eqref{def:ComparaisonPrinciple}. Let $0<\lambda < \underset{(t,x)\in M_T}{\inf}r(t,x)$, take $\psi^\lambda$ satisfying \eqref{hyp:EllipCoerc}, let $\delta > 0$, and, arguing by contradiction, suppose that:
        \begin{equation*}
                \sup_{(t,x) \in M_T} \underbrace{u(t,x) - \delta\left(\psi(x) + \frac{1}{t}\right)}_{\defeq J(\delta, t,x)}> 0.
        \end{equation*}
        Since, $u$ is bounded, and $\psi$ is coercive, and since they are continuous, we know that there exists a maximizer $(\Bar{t}, \Bar{x}) \in M_T$, since :
        \begin{equation*}
            \lim_{t \xrightarrow[]{} 0^+} J(\delta,t,x) = -\infty,
        \end{equation*}
        and the fact that $J(\delta, T, x)\leq 0, \ \forall \delta > 0, \ \forall x\in M$ ensures that $(\Bar{t}, \Bar{x}) \in (0,T)\times M$. 
        
        Now,
        \begin{align*}
            0  & < r(\Bar{t}, \Bar{x})\Big(u(\Bar{t}, \Bar{x}) - \delta\left(\psi(x) + \frac{1}{t}\right)\Big),\\
            & \leq (\partial_t u  + \Delta_g u)(\Bar{t}, \Bar{x})  - \delta r(\Bar{t}, \Bar{x})\left(\psi(x) + \frac{1}{t}\right),\\
            & \leq -\frac{\delta}{\Bar{t}^2} -  \frac{\delta r(\Bar{t}, \Bar{x})}{\Bar{t}}+ \Delta_g u(\Bar{t}, \Bar{x}) - \delta \Delta_g \psi(\Bar{x}) + \delta(\lambda - r(\Bar{t}, \Bar{x}))\psi(\Bar{x},)\\
            & \leq \Delta_g u(\Bar{t}, \Bar{x}) - \delta \Delta_g \psi(\Bar{x}) + \delta(\lambda - r(\Bar{t}, \Bar{x}))\psi(\Bar{x}).
        \end{align*}

       Using the first and second-order optimality condition on the Laplace Beltrami Operator in any local chart around $\Bar{t}, \Bar{x}$, one deduces that :
       
       \begin{equation*}
           \Delta_g u (\Bar{t}, \Bar{x}) - \delta \Delta_g \psi(\Bar{x}) = -g^{ij}\Gamma^k_{ij}\underbrace{\partial_i(u-\delta \psi^\lambda)}_{=0} + g^{ij}\partial_{ij}(u-\delta \psi^\lambda) \leq 0.
       \end{equation*}

        We can always suppose that $\psi^\lambda$ is positive because since it is coercive $(\inf_{M} \psi^\lambda ) \in \dR $ and thus,

        \begin{equation*}
             \widetilde{\psi}^\lambda(x) = \psi^\lambda(x)+ |\inf_{M} \psi^\lambda|,
        \end{equation*}

        still satisfies \eqref{hyp:EllipCoerc} and is positive.

        Finally,

        \begin{equation*}
            \delta(\lambda - r(\Bar{t}, \Bar{x}))\psi(\Bar{x}) \leq 0,
        \end{equation*}

        yields the desired contradiction. We conclude that,

        \begin{equation*}
            u(T,\cdot) \leq 0 \ \ \ \forall x \in M \implies u(t,x) \leq \delta(\psi^\lambda(x) + \frac{1}{t}) \ \ \ \ \forall \delta > 0, x \in M, t\in (0,T].
        \end{equation*}

        Passing to the limit as $\delta \xrightarrow[]{} 0^+$, yields the maximum principle.

        \item If the sectional curvature is everywhere negative, the injectivity radius is $+\infty$ everywhere, thus $\forall x_0 \in M$, $x\mapsto d^2(x_0,x)$ is $C^\infty(M)$ and $x\mapsto d(x_0,x)$ is $C^\infty(M \backslash \{x_0\})$. Define the following auxiliary function $\phi\in C^{2}_b(\dR_+,\dR_+)$ as,
        \begin{equation*}
          \phi (r) =    \begin{cases}
                            -\frac{r^4}{8} + \frac{3r^2}{4} \ \ &(0,1), \\
                            r - \frac{3}{8} \ \ &  [1,+\infty).\\
                        \end{cases}  
        \end{equation*}
        
        Take an arbitrary $x_0 \in M$, note $r(x) = d(x_0, x)$, and define $\psi(x) = \phi(r(x))$, then $\psi \in C^2(M)$, and:
        \begin{equation*}
            \Delta_g\psi(x) = \Delta_g r(x) \phi'(x) + |\nabla r(x)|^2 \phi''(r(x)),
        \end{equation*}
        using Theorem \ref{thm:LaplacianDistBoundedCurve}, we obtain:
        \begin{align*}
            -\Delta_g \psi &\geq -(n-1)K \phi'(r(x))coth(K r(x)) + |\nabla r(x)|^2 \phi''(r(x)),\\
            & \geq - C_K.
        \end{align*}
        
        Thus, for any $\lambda > 0 $, 

        \begin{equation*}
            \psi^\lambda(x) = \psi(x) + \frac{C_K}{\lambda},
        \end{equation*}

        satisfies \eqref{hyp:EllipCoerc}, we can conclude using 2.
        
       \item
            Fix some $x^* \in M$, and let $\Tilde{d}$ be the regularised distance function of Lemma \ref{lem:regulDist}, let $\psi : x \mapsto \Tilde{d}(x^*,x)$.
            For any $\lambda > 0$, define $\psi^\lambda(x) = \psi(x) + \frac{dC}{\lambda}$, where $C$ is the constant in the previous lemma for $n = 2$.
            Now note that $\psi^\lambda$ is positive, and coercive, in the sense of the third point. Finally, since
            \begin{equation*}
                \|\Delta_g \psi^\lambda\|_\infty \leq d\||\nabla^2 \psi^\lambda|_g\|_\infty\leq dC.
            \end{equation*}
            One can check that it is indeed a super solution of the desired elliptic equation.
            By applying the second point, we obtain the comparison principle.        
    \end{enumerate}
\end{proof}

\subsection{Weak flow solutions to the Fokker-Planck equation.}
We now introduce the concepts needed to define the notion of solution considered for the Fokker-Planck equation.

\begin{definition}[Wasserstein Spaces]
    Let $(M,d)$ be a complete separable metric space, note $\mathcal{B}$ its borelians. Define,
    \begin{equation*}
        \mathcal{P}_p(M) = \{ \mu \in \mathcal{P}(M), \exists x_0 \in M, \int_M d^p(x_0, x)\mu(dx) < +\infty  \},
    \end{equation*}
    the set of $p$-integrable probabilities on $(M,\mathcal{B})$, and define $\mathcal{W}_p$ as the p-Wasserstein distance on $\mathcal{P}_p(M)$ as,
    \begin{equation*}
        \mathcal{W}_p(\nu,\mu)^p = \inf_{\pi \in \Pi(\mu,\nu)} \int_{M\times M} d^p(x,y)\pi(dx,dy).
    \end{equation*}

    Then, $(\mathcal{P}_p(M), \mathcal{W}_p)$ is a complete metric space (see appendix for more details).
\end{definition}

\begin{definition}
    A \it{probability flow} solution associated to $B \in L^\infty([0,T],C^1_b(T^1M))$, is a probability flow $m \in C([0,T], \mathcal{P}(M))$, such that for any test function $\varphi \in C^{1,2}_b([0,T]\times M)$ and for all $t \in [0,T]$,
    \begin{equation*}
        \int_M \varphi(t,x)m_t(dx) - \int_M \varphi(0,x) m_0(dx) = \int_0^t\int_M (\partial_t \varphi + B \cdot \nabla_g \varphi + \frac{1}{2}\Delta_g \varphi)dm_s ds.
    \end{equation*}
\end{definition}

We now prove existence and uniqueness of such solution.
\begin{theorem}
    \label{thm:ProbaFlowSynthese}
     Let $(M,g)$ be a connected manifold of bounded geometry, suppose that the drift is in $L^\infty([0,T],C^1_b(T_1M))$, and that $m_0 \in \mathcal{P}_2(M)$.
     
     Then, there exists a weak solution in $C([0,T],\mathcal{P}_1(M))$, with Hölder estimate,
    \begin{equation*}
        \mathcal{W}_1(m_s,m_t) \leq C \sqrt{t-s},
    \end{equation*}
    and it is bounded in $\mathcal{P}_2(M)$, \textit{i.e} there exists $x_0 \in M$ such that,
    \begin{equation*}
        \sup_{t \in [0,T]} \int_M d^2(x_0,x)m_t (dx) < C,
    \end{equation*}
    with $C$ depending on the geometry of $M$, $\| |B|_g \|_\infty $, and $\int_M d^2(x_0,x)m_0(dx)$.
    
    Furthermore, if the drift is in $C_b^{(\alpha, \alpha/2)}([0,T]\times TM)$ for some $0<\alpha<1$, then the solution is unique.
\end{theorem}

\begin{proof}
    According to Lemma \ref{lem:existenceOMrdmvar} in the appendix, we can always construct a probability space to apply the existence result from~\cite{rakotomalala2024practical} with $A = Id_{TM}$ and yield existence of a process $(X_\cdot)$, projection on $M$ of a \textit{moving frame solution} $(U_\cdot)$ taking values on the orthonormal frame bundle. Applying Ito's Lemma, \textit{e.g} \cite{rakotomalala2024practical}, for any $\varphi \in C^{1,2}_b([0,T]\times M)$,
    \begin{equation*}
          \varphi(t, X_t) - \varphi(0, X_0) = \int_0^t(\partial_t \varphi + B \cdot \nabla_g \varphi + \frac{1}{2}\Delta_g\varphi)(u,X_u) du + M_t,
    \end{equation*}
    where $M_t$ is a martingale starting at zero.
    
    This gives,

    \begin{equation*}
        \dE[\varphi(t, X_t)] - \dE[\varphi(0, X_0)] = \dE[\int_0^t(\partial_t \varphi + B \cdot \nabla_g \varphi + \frac{1}{2}\Delta_g\varphi)(u,X_u) du].
    \end{equation*}

    Thus considering $m_t = \mathcal{L}(X_t)$, that is the law of the continuous process, we obtain $m\in C([0,T],\mathcal{P}(M))$ is a weak flow solution.
    Finally, using the stochastic flow estimate from \cite{rakotomalala2024practical}, one verifies the Hölder and moment estimates.
    
    Uniqueness is left to prove. Let $t \in [0,T]$ and take a test function $\varphi \in C_b^{(2+\alpha)}(M)$. Since the drift is in the space $C_b^{(\alpha, \alpha/2)}([0,T]\times TM)$, from Proposition \ref{th:parabolicReg}, there exists a solution $\psi \in C_b^{(\alpha+2, \alpha/2 + 1)}(M\times J)$ of :

    \begin{equation*}
        \begin{cases*}
            \partial_t \psi + B\cdot \nabla_g\psi + \Delta_g \psi = 0 \ \ \ \ M\times [0,t), \\
            \psi(t,\cdot) = \varphi.
        \end{cases*}
    \end{equation*}
    Take $m^1$ and $m^2$ two weak solutions, starting from the same probability measure $m_0$, we have :

    \begin{equation*}
        \int_M \psi(t,x) m^1_t(dx) = \int_M \psi(0,x) dm^1_0(dx) + \int_0^t\int_M (\partial_t \psi + B\nabla_g\psi + \Delta_g \psi)dm^1_s ds = \int_M \psi(0,x) dm^1_0(dx)
    \end{equation*}

    Since $\psi$ is a solution to a parabolic equation. We have the same equation with $m^2$, and since $m^1_0 = m^2_0$, we can conclude that, $\int_M \psi(t,x) m^1_t(dx) =\int_M \psi(t,x) m^2_t(dx)$, and by construction of $\psi$,
    \begin{equation}
        \label{eq:m1m2agree}
        \int_M \varphi dm^1_t =\int_M \varphi dm^2_t,
    \end{equation}
    for any $\varphi \in C_b^{2+\alpha}(M)$. We now prove why this implies that $m^1 = m^2$. Note that this would be immediate in $\dR^d$, but since we could not find a proper reference for this specific case we give the following details.
    
    Define $\mathcal{C}$ as, $\mathcal{C} = \{ O \text{ open}, \exists x \in M \text{ s.t } O \subset B(x, \epsilon/3) \}.$ It is straightforward to prove that it is a $\pi$-system. For any $O \in \mathcal{C}$, take any $x \in M$ such that $O \subset B(x,\epsilon/3)$, and take a normal coordinate system at $x$ up to $\epsilon$, it exists since we supposed that $\epsilon < i(M)$, and introduce the regularised sequence of function,
    \begin{equation*}
        \varphi^n = (\mathbbm{1}_{\{\cdot \in \phi(O)\}} * f^n) \circ \phi^{-1}, 
    \end{equation*}
    where $f^n$ is a sequence of mollifier in $\dR^d$ with support in $B_{\dR^d}(0,\epsilon/2)$, so that $\text{supp} \varphi^n \ssubset B(x,\epsilon)$, $\varphi^n$ can be extended in a $C^\infty$-manner to $M$ by setting it $0$ outside of $B(x,\epsilon)$. We also have that $\sup_n \|\varphi^n\|_\infty \leq 1$, the simple convergence $\varphi^n(x) \xrightarrow[n\mapto \infty]{} \mathbbm{1}_{x\in O}$, and $\varphi^n \in C^{\infty}_b(M)$. So that, we can use \eqref{eq:m1m2agree} for any $n$, and pass to the limit using dominated convergence theorem,
    and conclude that, $m^1_t(O) = m^2_t(O)$ for any $O$ in the $\pi$-system $\mathcal{C}$ and any $t\in [0,T]$. From Dynkin's $\pi-\lambda$-theorem, we obtain that $m^1_t$ and $m^2_t$ agree on the $\sigma$-algebra generated by $\mathcal{C}$. It rests to show, that this $\sigma$-algebra contains any open set.
    Let $\{O^\alpha\}_{\alpha \in \mathbb{N}}$ be a countable topological basis of $(M,g)$, we can assume that for any $\alpha \in \mathbbm{N}$, that the closure of $O^\alpha$ is compact, thus for any $\alpha$ there exists a finite collection $x_1,\cdots, x_{n_\alpha} \in M$, so that $\{B(x_i,\epsilon/2)\}_{i = 1}^{n_\alpha}$ covers $O^\alpha$, and define $\{ V^{\alpha, i}\}_{i = 1}^{n_\alpha}$ as,
    \begin{equation*}
        V^{\alpha, 1} = O^\alpha \cap B(x_1,\epsilon/2), \ \ V^{\alpha, k} = (O^\alpha \cap B(x_k,\epsilon/2)) \big/ (\bigcup_{i=1}^{k-1}V^{\alpha, i}).
    \end{equation*}
    So that, $\{V^{\alpha, i}\}$ is still a countable basis and $V^{\alpha,i} \in \mathcal{C}$. This implies that, $\sigma(\mathcal{C})$ coincide with the Borel $\sigma$-algebra. And thus, that $m^1_t = m^2_t$.
\end{proof}

\begin{lemma}(Compactness of Probability flow)
        \label{lem:CompactProbFlow}
        Suppose $(M,g)$ is a complete metric space, 
        \begin{equation*}
            \mathcal{K}(C_1) = \{ \mu\in C([0,T],\mathcal{P}_1(M)), \mathcal{W}_1(\mu_s,\mu_t) \leq C_1\sqrt{t-s}, \exists x_0 \in M, \int_M d^2(x_0,y)d\mu_t(y) < C_1\}
        \end{equation*}
        is convex and compact in $(C([0,T],\mathcal{P}_1(M)),\mathcal{D}(\mu,\nu) \defeq \sup_{s\in [0,T]}\mathcal{W}_1(\mu_s,\nu_s))$
\end{lemma}
\begin{proof}
    The result follows from Arzela-Ascoli Theorem and the compactness criterion Lemma \ref{lm:CompCritProb} of the appendix.
\end{proof}
\subsection{Analysis of second order mean field games with quadratic Hamiltonian.}

We now can formulate and prove the main theorem of this section.
    
\begin{theorem}
    Let M be a connected manifold of bounded geometry, and suppose that,
    \begin{assumptionp}{$\mathcal{H}_{MFG}$}
        $ m_0 \in \mathcal{P}_2(M), \exists \alpha \in (0,1), \exists C_0>0$, such that $ \forall \mu,\nu \in \mathcal{P}_1(M):$
        \begin{align*}
                G : \mathcal{P}_1(M) \xrightarrow{} C^{2+\alpha}(M),\ & \|G(\mu)\|_{C^{2+\alpha}} < C_0, & \|G(\mu) - G(\nu)\|_{C^{2+\alpha}} \leq C_0\mathcal{W}_1(\mu,\nu),\\
                F : \mathcal{P}_1(M) \xrightarrow[]{} C_b^{\alpha}(M), \ & \|F(\mu)\|_{C^{\alpha}} < C_0, & \|F(\mu) - F(\nu)\|_{C^{\alpha}} \leq C_0\mathcal{W}_1(\mu,\nu).\\
        \end{align*}
    \end{assumptionp}
    Then, the system \eqref{eq:MFG} admits a solution $(v,m)$, where $v\in C_b^{(1+\alpha/2, 2+\alpha)}(J\times M)$ solution in the classical sense of the Hamilton-Jacobi-Bellman equation, and $m\in C([0,T],\mathcal{P}_1(M))$ is a weak flow solution of the Fokker-Planck equation.
\end{theorem}

\begin{proof}
    \par
    We find a fixed point using Schauder's fixed point theorem.
    Let $(m_t)_t \in \mathcal{K}$ be a probability flow, ($\mathcal{K}$ defined as in Lemma \ref{lem:CompactProbFlow}),
    and consider the following equation :
    \begin{equation}
        \label{eq:HJB}
        \begin{cases*}
            -\partial_t u - \Delta_g u + \frac{1}{2}|\nabla u|_g^2 = F(m_t) \\
            u(T, \cdot) = G(m_T),
            \tag{HJB}
        \end{cases*}
    \end{equation}

    Introduce the following Cole-Hopf transformation $w = e^{-u/2}$ and observe that, since :
    \begin{equation*}
        \Delta_g w = \frac{1}{2}w(-\Delta_g u + \frac{1}{2}|\nabla u|^2_g),
    \end{equation*}
    then $w$ solves the following linear backward parabolic equation :
    \begin{equation*}
        \label{eq:linpara}
        \begin{cases*}
            \partial_t w + \Delta_g w - \frac{1}{2}F(m_t)w = 0,\\
            w(T,\cdot )= e^{-G(m_T,\cdot)/2}.
        \end{cases*}
    \end{equation*}

    The fact that $(m_t)_t \in \mathcal{K}$ ensures that: 
    \begin{equation*}
        a_0 \defeq ((x,t)\mapsto F(m_t,x)),
    \end{equation*}
     is in the space $ C^{\alpha/2,\alpha}(M\times J)$, thus Theorem \ref{th:parabolicReg} gives existence and uniqueness of a classical solution of \ref{eq:linpara}, in $C_b^{(1+\alpha/2, 2+\alpha)}(J\times M)$, and thus of the Hamilton-Jacobi-Bellman equation. 
    
    We can apply the maximum principle to $\Tilde{w} = we^{\frac{1}{2}C_0 t}$, since it is bounded by the estimate in Proposition \ref{th:parabolicReg}, and it is a solution of :
    \begin{equation*}
        -\partial_t \Tilde{w} - \Delta_g \Tilde{w} + \frac{1}{2}(a_0 + C_0)\Tilde{w} = 0,  
    \end{equation*}
    with $r = \frac{1}{2}(a_0 + C_0) > 0$, with the sub-solution $(t,x)\mapsto e^{-\frac{1}{2}C_0(T-2t + 1)}$and the super-solution $(t,x)\mapsto e^{\frac{1}{2}C_0(T-t + 1)}$, one can deduce that : 
    \begin{equation*}
        e^{-\frac{1}{2}C_0 (T-t+1)}\leq  w(t,x) \leq e^{\frac{1}{2}C_0 (T - t +1)}
    \end{equation*}

    and thus,

    \begin{equation*}
        \sup_{(t,x)\in M_T} |u(t,x)| \leq C_0(T+1)
    \end{equation*}

    We now want a uniform estimate on the gradient. By a density argument \cite{amann2017cauchy}, suppose that $a_0: t,x \mapsto F(m_t,x)$ is in $C^{(1+\alpha)/2,1+\alpha}([0,T]\times M)$, and that the terminal condition is in $C^{3+\alpha}(M)$, so that we can differentiate at the third order, according to Proposition \ref{th:parabolicReg}, the solution $w$ of the linear parabolic equation.
    From Bochner's formula, applied to $w$, one gets,
    
    \begin{equation*}
        \frac{1}{2}\Delta_g |\nabla w|_g^2 = \langle \nabla \Delta_g w, \nabla w \rangle_g + |\nabla^2 w|^2_g + Ric(\nabla w, \nabla w).
    \end{equation*}
    Using the fact that $w$ is a solution of a parabolic equation,
    \begin{align*}
        0 &= \langle \nabla (-\partial_t w + \frac{a_0}{2} w), \nabla w \rangle -  \frac{1}{2}\Delta |\nabla w|^2 + |\nabla^2 w|^2 + Ric(\nabla w, \nabla w)\\
        0 &= -\frac{1}{2} \partial_t |\nabla w|^2 + \frac{a_0}{2} |\nabla w |^2 + \frac{1}{2}w \langle \nabla a, \nabla w\rangle -  \frac{1}{2}\Delta |\nabla w|^2 + |\nabla^2 w|^2 + Ric(\nabla w, \nabla w)
    \end{align*}
    Since $M$ is of bounded geometry, its Ricci curvature is uniformly bounded below on M,
    \begin{equation*}
        -K_2 |\nabla w|^2 \leq Ric(\nabla w, \nabla w)
    \end{equation*}

    Using Bernstein's method with the auxiliary function $\phi = \frac{1}{2}|\nabla w|^2e^{\lambda t}$, with $\lambda$ to we be specified later, we note that,

    \begin{align*}
        -\partial_t \phi -\Delta_g \phi + \lambda \phi & = -e^{\lambda t}\Big( \frac{a_0}{2} |\nabla w |^2 + \frac{1}{2}w \langle \nabla a, \nabla w\rangle + |\nabla^2 w|^2_g + Ric(\nabla w, \nabla w)\Big)\\
        & \leq  - a_0 \phi +  e^{\lambda t}(\frac{1}{2}\|w\|_\infty \||\nabla a|_g\|_\infty |\nabla w|_g + K_2 |\nabla w|^2_g)\\
        & \leq (\|a_0\|_\infty + 2K_2)\phi + \frac{1}{2}e^{\lambda t}\|w\|_\infty \||\nabla a|_g\|_\infty (1 + |\nabla w|^2_g)
    \end{align*}        
    We obtain,
    \begin{equation*}
        -\partial_t \phi -\Delta_g \phi + \underset{\defeq r}{\underbrace{(\lambda - \|a_0\|_\infty - 2K_2 - \|w\|_\infty \||\nabla a|_g\|_\infty)}} \phi  \leq  \frac{1}{2}e^{\lambda t}\|w\|_\infty \||\nabla a|_g\|_\infty.
    \end{equation*}        
    Fixing $\lambda > \|a_0\|_\infty - 2K_2 - \|w\|_\infty \||\nabla a|_g\|_\infty $, so that $r > 0$, we can verify that,
    \begin{equation*}
        \varphi = \phi - \frac{1}{2r} \frac{1}{2}e^{\lambda T}\|w\|_\infty \||\nabla a|_g\|_\infty,  
    \end{equation*}
    is a sub-solution to the parabolic equation,
    \begin{equation*}
        - \partial_t \varphi - \Delta_g \varphi + r \varphi \leq 0.
    \end{equation*}
    Subtracting $\frac{1}{2}\||\nabla G|^2_g\|_\infty e^{C_0}/2$, since from the estimate in Proposition \ref{th:parabolicReg} $\varphi$ is bounded, we conclude from the maximum principle (see Proposition \ref{def:ComparaisonPrinciple}) that there exists C, only depending on $C_0$, $K_2$, $T$ such that,
    \begin{equation*}
         \||\nabla w|_g\|_\infty < C.
    \end{equation*}
    We conclude by density the desired  estimate, and using the representation of $u$, we deduce that $\| |\nabla u|_g \|_\infty < C$, only depending on $C_0$ and $K_2$ and $T$.

    Since $w$ and thus $v$ are classical solutions according to Proposition \ref{th:parabolicReg}, the drift $-\nabla u$ satisfies the hypothesis of Theorem \ref{thm:ProbaFlowSynthese}, thus there exists a unique solution to the Fokker Planck equation that we note $\Psi(m)$. Furthermore, from the uniform bound on the norm of the gradient, it follows from Theorem \ref{thm:ProbaFlowSynthese} that for any $m\in \mathcal{K}$, the solution $\Psi(m)$ stays in $\mathcal{K}$, if we pick $C_1$ large enough. Thus the map,
    \begin{equation*}
        \Psi : \mathcal{K} \subset C([0,T], \mathcal{P}_1(M) ) \xrightarrow{}  \mathcal{K}, \\
    \end{equation*}
    is well-defined. Next, we show that $\Psi$ is continuous.

    Let $\mu^n \in \mathcal{K}$ be a sequence converging to $\mu \in \mathcal{K}$, denote by $u^n$(resp. $u$) the solutions to HJB equation associated to the $\mu^n$(resp. $\mu$), and $m^n = \Psi(\mu^n)$ (resp. $m = \Psi(\mu)$). The mapping $\Psi$ takes value in the compact $\mathcal{K}$, so we can extract a sub-sequence (still noted $(m^n)_n$) converging to some $\nu \in \mathcal{K}$. Take $\varphi \in C^\infty_c(M_T)$, and define the differential operator as $\mathcal{A}[b]\varphi = b^i\partial_i \varphi + \frac{1}{2}\Delta_g \varphi$,
    then we have the following convergence, uniformly on any compact set,
    
    \begin{equation*}
        \lim_{n} \sup_K |(\partial_t \varphi + \mathcal{A}[-\nabla u^n]\varphi) - (\partial_t \varphi + \mathcal{A}[-\nabla u]\varphi)| = 0 \ \ \ \ \forall K \subset \subset M.
    \end{equation*} 

    Indeed, first we notice that,
    \begin{align*}
        |(\partial_t \varphi + \mathcal{A}[-\nabla u^n]\varphi) - (\partial_t \varphi + \mathcal{A}[-\nabla u]\varphi)| \leq |\nabla u^n - \nabla u|_g |\nabla \varphi |_g
    \end{align*}

    And since $\partial_{ij} u = \frac{\partial_{ij}w}{w} - \frac{\partial_iw \partial_j w}{w^2}$, we have that :

    \begin{equation*}
        (\nabla^2 u)_{ij} \eqcoord \partial_{ij}u -\Gamma^k_{ij}\partial _k u \eqcoord \frac{1}{w} (\partial_{ij}w - \Gamma^k_{ij}\partial_k w) -\frac{1}{w^2}\partial_iw \partial_j w \eqcoord \frac{1}{w}(\nabla^2 w)_{ij} - \frac{1}{w^2}\partial_iw \partial_j w 
    \end{equation*}
    
    Using that the norm of the $(0,2)-$tensor $(\partial_iw \partial_j w)$ is :
    
    \begin{equation*}
        g^{i_1j_1}g^{i_2 j_2} \partial_{i_1}w \partial_{j_1} w \partial_{i_2}w \partial_{j_2} w = |\nabla w|_g^4.     
    \end{equation*}

    And the triangular inequality on the tensor norm, we get that :
    \begin{equation*}
        |\nabla^2 u|_g \leq \frac{1}{w}|\nabla^2 w|_g + \frac{1}{w^2}|\nabla w|^2_g.
    \end{equation*}

    We deduce that, 
    
    \begin{equation*}
        \||\nabla^2 u^n|_g \|_\infty \leq C, 
    \end{equation*}

    uniformly in $n$, since $\|w\|_{bc^{2+s,1+s/2}}$ is uniformly bounded by a constant depending on the geometry and on $C_1$. This implies that the $(\nabla u^n)_n$ are uniformly continuous, and since they are uniformly bounded, by Arzela Ascoli theorem the sequence is compact on any compact subset of $M$, the simple convergence of the $\nabla u^n$ to $\nabla u$, is given by the continuity estimate in Proposition \ref{th:parabolicReg}. We obtain the uniform convergence on any compact subset of $M$ of $|\nabla u^n - \nabla u|_g$. Thus

    \begin{equation*}
        \Phi^n(t) = (\partial_t \varphi + \mathcal{A}[-\nabla u^n]\varphi)_t,
    \end{equation*}
    is uniformly continuous, bounded and converges uniformly on any compact subset of $M$ to 

    \begin{equation*}
        \Phi(t) = (\partial_t \varphi + \mathcal{A}[-\nabla u]\varphi)_t.
    \end{equation*}

    From Lemma \ref{lem:Talay}, defining $\Xi^n(t) = \int_M \Phi^n_t dm_t^n,$ and $\Xi(t) = \int_M\Phi_t d\nu_t$. We have the convergence for any $t\in [0,T]$ of $\Xi^n(t)$ to $\Xi(t)$.

    We can control $\Xi^n$ in time with,
    \begin{equation*}
        |\Xi^n(t) - \Xi^n(s)| \leq \|\Phi^n(t) - \Phi^n(s)\|_{M,\infty} + \mathcal{W}^1(m^n_t,m^n_s).
    \end{equation*}

    And since $\nabla^n u$ is uniformly bounded, and is in $BUC([0,T],C_b^\alpha(M))$, and $\varphi$ is regular, we have a uniform in time convergence of $\|\Phi^n(t) -\Phi^n(s)\|_{M,\infty}$. And $\mathcal{W}^1(m^n_t,m^n_s)$ is controlled by the estimate of theorem \ref{thm:ProbaFlowSynthese}, meaning that $(t \mapsto \Xi^n(t))$ is continuous in time, thus measurable, and uniformly bounded. From dominated convergence, we conclude that for any $t \in [0,T]$:
    \begin{equation*}
        \lim_n \int_0^t\int_M (\partial_t \varphi + \mathcal{A}[-\nabla^n u]\varphi)dm^n_s ds = \int_0^t\int_M (\partial_t \varphi + \mathcal{A}[-\nabla u]\varphi)d\nu_s ds.
    \end{equation*}

    And since the weak solution is unique: $\nu = m$. It is true for any subsequence, and the sequence takes values in the compact set $\mathcal{K}$, thus the sequence tends to $m$. Hence, $\Psi$ is continuous.
    
    Finally, from Schauder's fixed point theorem, there exists a solution to the system.
\end{proof}

    This study allows to highlight the technicalities encountered when studying mean field games systems on Riemannian manifolds. The analysis could be pursued further for Lipschitz Hamiltonian, and local couplings adapting the classical approach to the geometric case.

\section*{Acknowledgement}
The first author has been partially supported by the Chairs FDD (Institut Louis Bachelier) and FiME (Université Paris Dauphine - PSL et l’École Polytechnique) and the Lagrange Research Center in Paris.

\begingroup
\renewcommand{\thesubsection}{\Alph{subsection}}
\renewcommand{\thetheorem}{\thesubsection.\arabic{theorem}}
\renewcommand{\thedefinition}{\thesubsection.\arabic{definition}}
\renewcommand{\thedefinitionproposition}{\thesubsection.\arabic{definitionproposition}}
\renewcommand{\theproposition}{\thesubsection.\arabic{proposition}}

\appendix
\section*{Appendix}
\addcontentsline{toc}{section}{Appendix}

\subsection{Reminders in Riemannian geometry}
\begin{definition}[Smooth Manifold]
A d-dimensional manifold $M$ is second countable topological space, such that there exists an atlas $\{U_i, \phi_k\}_{i\in \mathfrak{K}}$, where the $U_i$ forms an open cover of $M$ and $\phi_i$ is a homeomorphism from $M$ to $\dR^n$. A couple $U, \phi$ is called a local chart.
\end{definition}

\begin{definition}[Tangent space]
Note $\mathcal{C}_x(M)$ the set of all smooth path $\gamma : \dR \supset U \xrightarrow[]{} M$ defined on $U$ a neigbourhood of $0$, s.t $\gamma(0) = x$. Define the following equivalent relation on $\mathcal{C}_x(M)$ :
\begin{align*}
    \gamma_1 \sim \gamma_2 & \Leftrightarrow \text{for all charts}\phi \text{ at }x, (\phi \circ \gamma_1)'(0) = (\phi \circ \gamma_2)'(0) \\
                            & \Leftrightarrow \text{there exist a chart}\phi \text{ at }x, (\phi \circ \gamma_1)'(0) = (\phi \circ \gamma_2)'(0) \\
\end{align*}
The tangent space $T_xM$ at x is defined as the quotient $T_xM = \mathcal{C}_x (M)/\sim$
\end{definition}

\begin{definition}[Vector Bundle]
A vector bundle of rank $k$ on a manifold $M$ consist of:
\begin{enumerate}
    \item a manifold $E$
    \item a continuous surjection $\pi : E \xrightarrow[]{} M$
    \item $\forall x \in M$, the fiber $\pi^{-1}(\{x\}) \eqdef E_x$ has a $k-$dimensional real vector space structure.
\end{enumerate}
Such that every point $x\in M$ admits a neighbourhood $U\subset M$ and an homeomorphism 
\begin{equation*}
    \phi : U \times \dR^k \xrightarrow[]{}\pi^{-1}(U),
\end{equation*}
such that $(\pi \circ \phi)(x,v) = x\ \ \ \ \forall v \in \dR^k$ and the map $v\mapsto \phi(x,v)$ is a linear isomorphism between the vector spaces $\dR^k$ and $\pi^{-1}(\{x\})$.
\end{definition}

\begin{definition}[Riemannian Metric]
A Riemannian metric $g$ assigns to each point $x\in M$ an inner product the tangent space at $x$, $g_x : T_xM \times T_xM \xrightarrow[]{} \dR $. In each charts, it is associated with a symmetric positive definite matrix $g_{ij}(x)$, which is supposed to be smooth.
The metric induces volume measure on $M$, $dvol = \sqrt{|\det g|}dx$.
\end{definition}

\begin{definition}[Connection]
    Let $\pi : E \mapto M$ be a smooth vector bundle over a smooth manifold M, and let $C^1(E)$ denote the space of $C^1$ sections of $E$. A \textit{connection} in $E$ is a map :

    \begin{equation*}
        \nabla : \mathcal{X}(M) \times C^1(E) \mapto C(E),
    \end{equation*}
    written $(X,Y)\mapsto \nabla_XY$, satisfying the following properties:
    \begin{enumerate}
        \item $\nabla_XY$ is linear over $C^\infty(M)$ in $X$, $\forall f_1, f_2 \in C^\infty(M)$ and $X_1,X_2 \in \mathcal{X}(M)$,
        \begin{equation*}
            \nabla_{f_1X_1+f_2X_2}Y =f_1\nabla_{X_1}Y+f_2\nabla_{X_2}Y
        \end{equation*}
        \item $\nabla_XY$ is linear over $\dR$ in $Y$, $\forall a_1, a_2 \in \dR$ and $Y_1, Y_2 \in C^1(E)$,
        \begin{equation*}
            \nabla_X(a_1Y_1+a_2Y_2) =a_1\nabla_XY_1+a_2\nabla_XY_2.
        \end{equation*}
        \item  $\nabla$ satisfies the following product rule, $\forall f \in C^\infty(M)$,
        \begin{equation*}
            \nabla_X(fY) = f\nabla_XY+(Xf)Y.
        \end{equation*}
    \end{enumerate}
\end{definition}

\begin{theorem}[Fundamental Theorem of Riemannian Geometry]
    Let $(M,g)$ be a Riemannian Manifold, there exist a unique connection that is \textbf{compatible} with the metric and \textbf{symetric}. It is called the Levi-Cevita connection. 
    
    In chart using the Christofel symbol, it writes,
    \begin{equation*}
        \Gamma^k_{ij} = \frac{1}{2} g^{kl} (\partial_i g_{jl} + \partial_j g_{il} - \partial_l g_{ij})
    \end{equation*}
\end{theorem}


\begin{definition}[Riemann Curvature Tensor]
    Let $X,Y,Z\in C^2(TM)$ be three vector fields on $M$. We define the \textit{Riemann curvature tensor} with the following formula,
    \begin{equation*}
        R(X,Y)Z = \nabla_X\nabla_Y Z - \nabla_Y\nabla_X Z - \nabla_{[X,Y]}Z,
    \end{equation*}
    where $[X,Y]$ is the Lie bracket of the vector fields.
    In coordinates it writes,
    \begin{equation*}
        R^l_{ijk} = \partial_j\Gamma^l_{ik} - \partial_k \Gamma^l_{ij} + \Gamma^l_{jm} \Gamma^m_{ik} - \Gamma^l_{km} \Gamma^m_{ij}
    \end{equation*}
\end{definition}

From the Riemann curvature tensor we can introduce, the \textit{Ricci curvature} tensor, it is defined as the contraction of the Riemann tensor. In coordinate it writes,
\begin{equation*}
    \text{Ric}_{ij} = R^k_{ikj}.
\end{equation*}
We say that the Ricci curvature tensor is bounded below, if there exists a scalar $\lambda > 0$, such that $\forall X\in C(TM)$,
\begin{equation*}
    -\lambda g(X,X) \leq \text{Ric}(X,X) \text{ (resp. }\text{Ric}(X,X) \leq \lambda g(X,X) \text{)},
\end{equation*}
or equivalently in coordinates, $-\lambda g_{ij} \leq \text{Ric}_{ij}$ in the sens of positive definite matrices.

We also introduce, for $u, v \in T_{\{x\}}M$ two linearly independent tangent vectors in the tangent space at a point $x$, the \text{sectional curvature} as,
\begin{equation*}
    \frac{g_x(R_x(u,v),u)}{g_x(u,u)g_x(v,v) - g_x(u,v)^2}.
\end{equation*}

Finally, recall that the covariant derivative of a tensor is induced from the connection on the tangent space, it is defined as follows, for a tensor field $A \in C^1(T^n_mM)$, the covariate derivative is a $(n,m+1)$-tensor, defined in coordinate as,
\begin{align*}
    (\nabla A) \eqcoord \partial_i a^{k_1 \cdots k_n}_{j_1 \cdots j_m} & + \Gamma^{k_1}_{il}a^{lk_2 \cdots k_n}_{j_1 \cdots j_m} + \ldots + \Gamma^{k_n}_{il}a^{k_1 \cdots k_{n-1}l}_{j_1 \cdots j_m}  - \Gamma^{l}_{j_1 i} a^{k_1 \cdots k_n}_{lj_2 \cdots j_m} -\ldots - \Gamma^{l}_{j_n i} a^{k_1 \cdots k_n}_{j_1 \cdots j_{m-1}l}.
\end{align*}

\subsection{Probability Measure Spaces}
The following lemmata are classical, for the sake of completeness, we give their proof in the current setting.
\begin{lemma}
    \label{lem:Talay}
    Let $(E,d)$ metric separable complete (Polish space), note $\mathcal{B}$ its Borelians. Let $\Prob_n, \Prob\in \mathcal{P}(E,\mathcal{B})$, if $\Prob_n$ converge weakly to $\Prob$ and if $\Phi^n$ converges uniformly on any compact set to $\Phi$ and the $\Phi^n$ are uniformly bounded. Then, 
    \begin{equation*}
        \lim_{n}\int_E \Phi_n d\Prob_n = \int_E \Phi d\Prob.
    \end{equation*}
\end{lemma}

\begin{proof}
    Since $(\Prob_n)$ converges weakly, it is tight (from Prokhorov Theorem) thus :
    $\forall \epsilon > 0, \exists K_\epsilon$ a compact set of $E$ such that $\sup_n \Prob_n(K_\epsilon^c) \leq \epsilon$.
    
    \begin{align*}
        \int_E \Phi d\Prob - \int_E \Phi_n d\Prob^n   & = \underbrace{\int_E \Phi d\Prob -\int_E \Phi d\Prob_n}_{\to 0,\text{weak convergence}} + \int_E \Phi d\Prob_n - \int_E \Phi^n d\Prob_n \\
    \end{align*}
    
    Since the sequence is uniformly bounded, $\Phi_\infty$ is bounded and the first term tends to 0. For the second term, we note that,
    \begin{align*}
        \left|\int (\Phi - \Phi^n) d\Prob_n\right| &\leq \left|\int_{K_\epsilon} (\Phi - \Phi^n) d\Prob_n \right| + \left|\int_{K_\epsilon^c} (\Phi - \Phi^n) d\Prob_n\right|,\\ & \leq \sup_{K_\epsilon}|\Phi - \Phi^n| \Prob_n(K_\epsilon) +  \epsilon(||\Phi||_\infty +||\Phi^n||_\infty),\\
        & \leq \sup_{K_\epsilon}|\Phi - \Phi^n| +  \epsilon(||\Phi||_\infty +\sup_k||\Phi^k||_\infty).
    \end{align*}
    Taking the limit in $n$ to infinity, and then in $\epsilon$ to zero, we obtain the desired result.
\end{proof}

\begin{lemma}[Compactness Criterion]
    \label{lm:CompCritProb}
    Let $r\geq p > 0$ and $\mathcal{K} \subset \mathcal{P}_p(M) $ be s.t $\exists x_0 \in M$,
    \begin{equation*}
        \sup_{\mu \in \mathcal{K}} \int_M d^r(x_0,x) d\mu(x) < +\infty
    \end{equation*}

    Then, $\mathcal{K}$ is tight(for the weak topology). If moreoever $r>p$, then $\mathcal{K}$ is relatively compact for the $d_p$ distance.
\end{lemma}

\begin{proof}
    Let $\epsilon > 0$ and $R > 0$ sufficiently large. We have for any $\mu \in \mathcal{K}$:

    \begin{equation*}
        \mu(M \symbol{92} B_R(x_0)) \leq \int_{d(x_0,x)\geq R} \frac{d(x_0,x)^r}{R^r} d\mu(x) \leq \frac{C}{R^r} < \epsilon
    \end{equation*}

    where $C = \sup_{\mu \in \mathcal{K}} \int_M d(x_0,x)^r d\mu < \infty$. So $\mathcal{K}$ is tight.
    From Prokhorov Theorem, we deduce that $\mathcal{K}$ is  sequentially compact for the weak comvergence. Let $(\mu_n)$ a sequence weakly converging to $\mu$.
    If we notice that,
    \begin{equation*}
        \int_{d(x_0,x)\geq R} d^p(x_0,x) d\nu(x) \leq \frac{1}{R^{r-p}}\int_{d(x_0,x)\geq R} d^r(x_0,x) d\nu(x) \leq \frac{C}{R^{r-p}}
    \end{equation*}

    holds for any probability in $\mathcal{K}$, in particular for the sequence. This gives us that :

    \begin{equation*}
        \lim_{R\xrightarrow[]{} \infty} \limsup_{n \xrightarrow[]{}\infty} \int_{d(x_0,x)\geq R} d^p(x_0,x) d\mu_n(x) = 0
    \end{equation*}

     and we conclude from  Theorem 7.2 p212 \cite{villani2021topics} that the sequence actually converges for the Wasserstein-p metric.
\end{proof}

Finally, we prove the following lemma required for the existence of weak flow solutions.

\begin{lemma}
    \label{lem:existenceOMrdmvar}
    Let $(M,g)$ be a complete connected n-dimensional Riemannian Manifold. Given a random variable $X_0$ defined on a probability space $(\Omega, \mathcal{F}, \Prob)$ taking values in M.
    
    Then, there exists a random variable $U_0$ on $(\Omega, \mathcal{F}, \Prob)$ taking value in the orthogonal frame bundle $OM$, such that $\pi(U_0) = X_0$, where $\pi$ is the projection on $M$.
\end{lemma}

\begin{proof}
    We are going to use Kuratowski and Ryll-Nardzewski measurable selection theorem. Set $\psi(\omega) = \pi^{-1}(X_0(\omega))$ is a closed non-empty subset of $OM$.

    $OM$ is a Polish space since M is connected and complete this implies that it is complete, it is metrizable with the induced metric from the metric on M, and noting that the $(\pi^{-1}(B(x,n)))_{n\in \mathbb{N} }$ are compact, where $B(x,n)$ is the ball of radius $n$ centered at some reference point $x\in M$, gives that it is locally compact and countable at infinity since the orthogonal frame bundle over a compact manifold (restriction of the manifold to the balls) is compact.

    Now take an open set $\mathcal{U}$ of $OM$, we need to show the weak measurability of $\psi$.

    \begin{align*}
        \{\omega, \psi(\omega)\cap\mathcal{U} \neq \varnothing\} & = \{\omega, \pi^{-1}(X_0(\omega))\cap\mathcal{U} \neq \varnothing \}\\
         & = \{\omega, X_0(\omega)\in \pi^{-1}(\mathcal{U})\}.
    \end{align*}

    \begin{equation*}
        \pi^{-1}(\pi(\mathcal{U})) = \bigcup_{g \in O(n)} g\mathcal{U}.
    \end{equation*}

    Right is open, so since $OM$ is equipped with the final topology, $\pi(\mathcal{U})$ is open. And thus 

    \begin{equation*}
        \{\omega, X_0(\omega)\in \pi^{-1}(\mathcal{U})\} \in \mathcal{F}.
    \end{equation*}

    Thus $\psi$ is weakly measurable, by Kuratowski theorem, there exists a measurable selection $U_0$, that is a random variable taking values in $OM$ and by construction $\pi(U_0) = X_0$.
\end{proof}

\endgroup
\newpage

\begingroup
\sloppy
\printbibliography
\endgroup

\end{document}